\documentclass[11pt]{amsart}

\usepackage[T1]{fontenc}

\usepackage[utf8x]{inputenc}

\usepackage[pagewise]{lineno} %\linenumbers

\usepackage{theoremref}
\usepackage{mathrsfs}
\usepackage{hyperref}

\usepackage{tikz}
\usepackage{times}
\usepackage{color}
\usepackage{amsmath}
\usepackage{amssymb}
\usepackage{amsthm}
\usepackage{enumerate}
\usepackage{amsbsy}
\usepackage{amsfonts}
\newtheorem{theo}{Theorem}[section]
\newtheorem{defin}[theo]{Definition}
\newtheorem{prop}[theo]{Proposition}
\newtheorem{coro}[theo]{Corollary}
\newtheorem{lemm}[theo]{Lemma}
\newtheorem{rem}[theo]{Remark}

\newcommand{\al}{\alpha}
\newcommand{\be}{\beta}

\newcommand{\Ga}{\Gamma}

\newcommand{\om}{\omega}
\newcommand{\Om}{\Omega}

\newcommand{\ep}{\epsilon }
\newcommand{\te}{\theta}
\newcommand{\De}{\Delta}
\newcommand{\de}{\delta}

\newcommand{\pa}{\partial}

\newcommand{\R}{{\mathbb R}^n}
\newcommand{\hR}{{\mathbb R}^n_+}

\newcommand{\ri}{\rightarrow}
\newcommand{\Rn}{{\mathbb R}^{n-1}}

\newcommand{\na}{\nabla}

\newcommand{\calC}{{ \mathcal C }}

\newcommand{\bke}[1]{\left( #1 \right)}

\newcommand{\bket}[1]{\left\{ #1 \right\}}
\newcommand{\norm}[1]{\left\Vert #1 \right\Vert}
\newcommand{\abs}[1]{\left| #1 \right|}
\newcommand{\rabs}[1]{\left. #1 \right|}

\newcommand{\calS}{{ \mathcal S  }}
\newcommand{\calT}{{ \mathcal T  }}
\newcommand{\calL}{{ \mathcal L  }}
\newcommand{\calB}{{ \mathcal B  }}
\newcommand{\calN}{{ \mathcal N  }}

\begin{document}
\baselineskip=15pt

\title[Local regularity near boundary]{Local regularity near boundary for the Stokes and Navier-Stokes equations}

\maketitle

\begin{center}
  \normalsize
%  \authorfootnotes
  Tongkeun Chang \footnote{Department of Mathematics,
   Yonsei University Seoul, 03722,
    South Korea, \,\, chang7357@yonsei.ac.kr } and Kyungkeun Kang \footnote{Department of Mathematics,
   Yonsei University Seoul, 03722,
    South Korea, \,\, kkang@yonsei.ac.kr}%\textsuperscript{2}
 % \par
 % \vspace{3mm}%\footnote{Author A}\textsuperscript{1},
  %\textsuperscript{1}
 % {Department of Mathematics,\\
  % Yonsei University Seoul, 136-701,\\
  %  South Korea,}

\end{center}

\begin{abstract}
We are concerned with local regularity of the solutions for the Stokes and Navier-Stokes equations near boundary. Firstly, we construct a bounded solution but its normal derivatives are singular in any $L^p$ with $1<p$ locally near boundary. On the other hand, we present
criteria of solutions of the Stokes equations near boundary to imply that the gradients of solutions are bounded (in fact, even further H\"{o}lder continuous).
Finally, we provide examples of solutions whose local regularity near boundary is optimal.
% in comparison to previously known results.

\noindent 2020  {\em Mathematics Subject Classification.}  primary
35Q30,
secondary 35B65.

\noindent {\it Keywords and phrases: Stokes equations, Navier-Stokes equations,  local regularity near boundary}

\end{abstract}

%\maketitle

\section{Introduction}
\setcounter{equation}{0}

We consider  the non-stationary Stokes equations near flat boundary
\begin{equation}\label{Stokes-10}
u_t - \De u +  \na \pi =0\qquad \mbox{div } u =0 \quad \mbox{ in }
\,B^+_{2}\times (0, 4),
\end{equation}
where  $B^+_{r}:=\{ x=(x', x_n )\in \Rn \times {\mathbb R}: |x|<r, x_n >0\}$. Here, no-slip boundary condition is given only on the flat boundary, i.e.
\begin{equation}\label{bddata-20}
u=0 \quad \mbox{ on } \,\Sigma:=(B_{2}\cap\{x_n=0\})\times (0,4).
\end{equation}

We can also consider similar situation for the Navier-Stokes equations, i.e.
\begin{equation}\label{nse-30}
u_t - \De u +  (u\cdot\nabla) u+\nabla  \pi =0,\qquad \mbox{div } u =0 \quad \mbox{ in }
\,B^+_{2}\times (0, 4)
\end{equation}
with the boundary condition \eqref{bddata-20}.

Our concern is local analysis of the solutions of the equations \eqref{Stokes-10} or \eqref{nse-30} with \eqref{bddata-20} in $B^+_{1}\times (0, 1)$. Unlike the heat equation, non-local
effect of the Stokes equations may cause limitation of local smoothing effects of solutions.
In fact, the second author showed that there exists a weak solution of the Stokes equations \eqref{Stokes-10}-\eqref{bddata-20} whose normal derivative is unbounded near boundary, although
it is bounded and its derivatives are square integrable (see \cite{Kang05}), i.e.
\[
\sup_{Q_{\frac{1}{2}}^+}\abs{D_{x_n} u}=\infty, \qquad \sup_{Q_1^+}\abs{u}+\norm{\nabla u}_{L^2(Q_1^+)}<\infty,
\]
where $Q^+_r : = B_r^+ \times (1-r^2,1), \,\, 0 < r < 1$.

Seregin and \u{S}ver\'ak found a simplified example as the form of shear flow in a half-space to the Stokes equations (and Navier-Stokes equations as well) such that its gradient is unbounded near boundary, although velocity field is locally bounded  (see \cite{Seregin-Sverak10}). It is noteworthy that the example in \cite{Seregin-Sverak10} is not of globally finite energy, and on the other hand, the constructed one in \cite{Kang05} has finite energy globally in a half-space (see \cite{KLLT21} for the details).

Following similar constructions as in \cite{Kang05},
even further, the authors  constructed an example that $u$ is integrable in $L^4_tL^p_x$ but $\nabla u$ is not square integrable (see \cite{CKca}). More precisely, there exists a very weak solution of the Stokes equations \eqref{Stokes-10} or the Navier-Stokes equations \eqref{nse-30} with \eqref{bddata-20} such that
\[
\norm{\nabla u}_{L^2(Q_1^+)}=\infty,\qquad \norm{u}_{ L^{4}_{t} L^p_{x}(Q_1^+)}<\infty,\quad p<\infty.
\]
The notions of very weak solutions are given in Section 2 (see Definition \ref{stokesdefinition} and Definition \ref{NS-vweak}).
With the aid of this construction, the authors also proved that Caccioppoli's inequalities of Stokes equations and Navier-Stokes equations in general may fails near boundary
when only local boundary problems are considered (see \cite[Theorem 1.1]{CKca}). It is a very important distinction in comparison to the interior case, where Caccioppoli's (type) inequalities turn out to be true (compare to \cite{J} and \cite{Wolf15}, and refer also \cite{Jin-Kang17} for generalized Navier-Stokes flow).

One may ask how bad $\nabla u$ could be, when $u$ is bounded in a local domain near boundary.
One of our motivations in this paper is to answer to the question, and we obtain the following:
\begin{theo}\label{maintheorem-SS}
Let $1 < p < \infty$. Then, there exists a very weak solution  $u$ of  Stokes equations \eqref{Stokes-10}  or Navier-Stokes equations \eqref{nse-30} with the boundary condition \eqref{bddata-20} such that
\begin{equation}\label{blowup-40}
\| u \|_{L^\infty(Q_1^+)}<\infty, \qquad \|\na  u \|_{L^{p}(Q^+_\frac12 )} =\infty.
\end{equation}
\end{theo}

\begin{rem}
If we compare to the example constructed by Seregin and \u{S}ver\'ak, their solutions also show
singular normal derivatives near boundary not only pointwisely but also $L^p_{\rm loc}$, $p>3$, since $\partial_{x_3}u(x_3, t)\ge cx_3^{-1+2\alpha}$, $\alpha\in (0,1/2)$ in the region near origin with $x_3^2\ge -4t$. Theorem \ref{maintheorem-SS} is an improvement of their result, since construction of singular normal derivatives in $L^p_{\rm loc}$ is extended up to all $p>1$ near boundary.
\end{rem}

\begin{rem}
\label{rem-thm11}
We do not know if $p$ in \eqref{blowup-40} can be replaced by $1$ in Theorem \ref{maintheorem-SS}, and thus we leave it as an open question.
In Appendix \ref{alter-thm11}, alternative proof is given for $p=2$ in Theorem \ref{maintheorem-SS}, which seems informative.
\end{rem}

On the other hand, the second motivation of the paper is to study optimal regularity
of the local problem near boundary, in case that the pressure $\pi$ is locally integrable in $L^p$. We recall that it was shown in \cite[Proposition 2]{Seregin00} in three dimensions that
for given $p, q\in (1,2]$ the solution of the  Stokes equations \eqref{Stokes-10}-\eqref{bddata-20} satisfies the following a priori estimate:
For any  $r$ with $p\le r<\infty$
\begin{equation}\label{Seregin-50}
\norm{u_t}_{L^{q}_t L^r_{x}(Q_{\frac{1}{4}}^+)}+\norm{\nabla^2 u}_{L^{q}_t L^r_{x}(Q_{\frac{1}{4}}^+)}+\norm{\nabla \pi}_{L^{q}_t L^r_{x}(Q_{\frac{1}{4}}^+)}
\le C\bke{\norm{ u}_{L^{q}_t W^{1,p}_{x}(Q_{1}^+)}+\norm{\pi}_{L^{q}_t L^p_{x}(Q_{1}^+)}}.
\end{equation}
Furthermore, due to parabolic embedding, it follows that
\begin{equation}\label{Seregin-100}
\norm{u}_{\calC^{\frac{\alpha}{2}}_t \calC^{\alpha}_{x}(Q_{\frac{1}{4}}^+)}
< \infty,\qquad  0< \alpha <2-\frac{2}{q}.
\end{equation}
We remark that there are examples such that the H\"{o}lder continuity \eqref{Seregin-100} is optimal. To be more precise, in case that $q<2$, it was proved in \cite{KLLT21} that
\begin{equation}\label{Seregin-200}
\norm{\nabla u}_{L^{\infty}(Q^+_\frac14 )}=\infty.
\end{equation}
\begin{rem}\label{rem0711-2}
For the case $q=2$, we can also show that \eqref{Seregin-200}
is true, and the details of its verification will be given in Appendix \ref{appendix0131-0}.
\end{rem}

The above estimate \eqref{Seregin-50} shows that integrability of $u_t$, $\nabla^2 u$ and $\nabla \pi$
is increased for spatial variables, and it is, however,  not clear if integrability in time could be improved or not.

Firstly, we extend the result of \cite[Proposition 2]{Seregin00} for the case $q>2$. In such case, an interesting feature is that not only velocity but also the gradient of velocity fields are H\"{o}lder continuous up to boundary, contrary to the case $1<q\le 2$.

\begin{theo}\label{Stokes-maximal}
Let $ 2< q < \infty $ and $1 < p < \infty$.
Suppose that   $(u,\pi)$ is solution for the Stokes equations  \eqref{Stokes-10}-\eqref{bddata-20} satisfying $\na ^2 u, \,\, u_t\in L^{q}_t L^p_x(Q_1^+)$ and $\pi\in L^q_{t}W^{1,p}_x(Q_1^+)$.
Then, for any $r$ with $p\le r<\infty$
\begin{equation}\label{SS-max-10}
\norm{u_t}_{L^{q}_t L^r_{x}(Q_{\frac{1}{4}}^+)}+\norm{\nabla^2 u}_{L^{q}_t L^r_{x}(Q_{\frac{1}{4}}^+)}+\norm{\nabla \pi}_{L^{q}_t L^r_{x}(Q_{\frac{1}{4}}^+)}
\le c\bke{\norm{ u}_{L^{q}_t W^{1,p}_{x}(Q_{1}^+)} + \norm{\pi}_{L^{q}_t L^p_{x}(Q_{1}^+)}}.
\end{equation}
Furthermore, the derivative of $u$ is H\"{o}lder continuous, i.e.
\begin{equation}\label{SS-max-20}
\norm{\nabla u}_{\calC^{\frac{\al}{2}}_t \calC^{\al}_{x}(Q_{\frac{1}{4}}^+)}
< \infty,\qquad  0 < \al < 1-\frac{2}{q}.
\end{equation}
\end{theo}

In next theorem, we prove   that the estimates \eqref{Seregin-50} and \eqref{SS-max-10} are indeed optimal. To be more precise, there is a solution of the Stokes equations \eqref{Stokes-10}-\eqref{bddata-20} such that $\nabla u$ and $\pi$ belong to $L^p(Q_{1}^+)$ but
$u_t, \nabla \pi\notin L^{\tilde{p}}(Q_{\frac{1}{2}}^+)$ for any $\tilde{p}$ with $\tilde{p}>p$.

\begin{theo}\label{maintheorempressure-SS}
Let $1 < p, \, q <  \infty$.
Then, there exist   a solution
 $u$ of  Stokes equations \eqref{Stokes-10} and Navier-Stokes equations \eqref{nse-30} with the boundary condition \eqref{bddata-20} such that
\begin{align}\label{0711-1}
\norm{\na^2 u}_{L^{\frac{3q}2}(Q_{1}^+)}+\norm{ D_t u}_{L^{q}(Q_{1}^+)}+\norm{\nabla \pi}_{L^{q}(Q_{1}^+)}<\infty,
\end{align}
but for any  $r_1>q$ and $r_2>\frac{3q}{2}$
\begin{align}\label{0711-2}
 \| D_t u \|_{L^{r_1} (Q_{\frac{1}{2}}^+)} =\infty, \qquad \| \nabla \pi \|_{L^{r_1} (Q_{\frac{1}{2}}^+)} =\infty, \qquad \| \nabla^2 u \|_{L^{r_2} (Q_{\frac{1}{2}}^+)} =\infty.
\end{align}
%If $q > \frac{2(n+2)}{n+8}$,
%  there exists a weak solution  $u$ of   Navier-Stokes equations \eqref{nse-30} with the boundary condition \eqref{bddata-20} such that $u$ satisfies \eqref{0711-1} and \eqref{0711-2}.
\end{theo}

\begin{rem}
We do not know
whether or not there exists a solution $u$ of  Stokes equations  \eqref{Stokes-10}-\eqref{bddata-20} such that $\| \nabla^2 u \|_{L^{r_1} (Q_{\frac{1}{2}}^+)} =\infty$ for $r_1>q$. In fact, our construction shows that $\nabla^2 u\in L^{\frac{3q}{2}} (Q_{\frac{1}{2}}^+)$.
\end{rem}

This paper is organized as follows.
In Section \ref{prelim}, we introduce the function spaces and we recall some
known results and introduce results useful for our purpose.
Section \ref{SS-half} is devoted to recalling Poisson kernel for Stokes equations in a half-space and two useful  lemmas are proved as well.
In Section \ref{thm1-1ss}, Section \ref{thm1-4ss} and Section \ref{proofss}, we present
the proofs of Theorem \ref{maintheorem-SS}, Theorem \ref{Stokes-maximal} and Theorem \ref{maintheorempressure-SS} for the Stokes equations, respectively.
In the case of the Navier-Stokes equations, proofs of Theorem \ref{maintheorem-SS} and Theorem \ref{maintheorempressure-SS} are given in Section \ref{proof-ns}.
Appendix provides a simple proof of Theorem \ref{maintheorem-SS} for the case $p=2$, and presents
the details of Remark \ref{rem-thm11}, Remark \ref{rem0711-2} and Remark \ref{rem0314} as well.

\section{Preliminaries}
\label{prelim}
\setcounter{equation}{0}

For notational convention,  we denote $x=(x^{\prime},x_n)$, where the symbol
${\prime}$ means the coordinate up to $n-1$, that is,
$x^{\prime}=(x_1,x_2,\cdot\cdot\cdot,x_{n-1})$. We write $D_{x_i} u$
  as the partial derivative of $u$ with respect to $x_i, \,\, 1 \leq i \leq n$, i.e., $D_{x_i} u(x) =
\frac{\pa }{\pa x_i}u(x)$. Throughout this paper we denote by $c$ various generic positive constant and by $c(*,\cdots,*)$ depending on the quantities in the
parenthesis.

%For $\al \in {\mathbb R}$, we
%consider a distribution $h_\al $  whose Fourier transform
%in ${\mathbb R}^{n+1}$ is defined by
%\begin{eqnarray*}
%\widehat{ h_{\al}} (\xi,\tau) = c_\al(  4\pi^2 |\xi|^2 + i
%\tau)^{-\frac{\al}{2}}, \quad (\xi, \tau) \in {\mathbb R}^n \times {\mathbb R}.
%\end{eqnarray*}
Let $\al \in {\mathbb R}$ and $1\leq p \leq \infty$. We define an
anisotropic homogeneous Sobolev space $\dot W^{\al, \frac{\al}2}_p ({\mathbb R}^{n+1})$ by
\begin{eqnarray*}
\dot W^{\al, \frac{\al}2}_p  ({\mathbb R}^{n+1}) = \{ f \in {\mathcal S}'({\mathbb
R}^{n+1}) \, | \, f  = h_{\al} * g, \quad \mbox{for some} \quad g   \in L^p ({\mathbb R}^{n+1}) \}
\end{eqnarray*}
with norm
%\begin{align*}
$\|f\|_{\dot W^{\al, \frac{\al}2}_p  ({\mathbb R}^{n+1})} : = \| g \|_{L^p({\mathbb R}^{n+1})} \,  = \|  h_{-\al} * f  \|_{L^p({\mathbb R}^{n+1})},$
%\end{align*}
 where
$*$ is a convolution in ${\mathbb R}^{n+1}$ and ${\mathcal S}^{'}({\mathbb
R}^{n+1})$
 is the dual space of the Schwartz space
${\mathcal S}({\mathbb R}^{n+1})$.
Here $h_\al $  is a distribution whose Fourier transform
in ${\mathbb R}^{n+1}$ is defined by
\begin{eqnarray*}
\widehat{ h_{\al}} (\xi,\tau) = c_\al(  4\pi^2 |\xi|^2 + i
\tau)^{-\frac{\al}{2}}, \quad (\xi, \tau) \in {\mathbb R}^n \times {\mathbb R}.
\end{eqnarray*}
 In case that $\alpha=k \in {\mathbb N} \cup \{ 0\}$, we note that
\begin{align*}
 \|f\|_{\dot W^{k, \frac{k}2}_p  ({\mathbb R}^{n+1})  } & \approx \sum_{k_1 + k_2 + \cdots k_n + l = k} \| D_{x_l}^{k_1}D_{x_2}^{k_2} \cdots D_{x_n}^{k_n}D^{\frac12 l}_{t}   f\|_{L^p ({\mathbb R}^{n+1})},
\end{align*}
where $\displaystyle D^\frac12_t f(t) = D_t \int_{-\infty}^t \frac{f(s)}{(t -s)^\frac12} ds$ and $D^{k  +\frac12}_tf = D^\frac12_t D^k_tf$. In particular, when $ k =0$, it follows that $\displaystyle\| f\|_{\dot W^{0,0}_p({\mathbb R}^{n+1})}  = \| f\|_{ L^p({\mathbb R}^{n+1})}$.

Next, we recall an anisotropic homogeneous Besov space $\dot B^{\al, \frac{\al}2}_{pq} ({\mathbb R}^{n+1})$. Let $\phi \in {\mathcal
S} ({\mathbb R}^{n+1})$ such that $\hat \phi$, the Fourier transform of $\phi$.  satisfies
\begin{eqnarray*}
 \left\{\begin{array}{rl}
&\displaystyle\hat \phi(\xi,\tau) > 0  \quad \mbox{ on }  2^{-1} < |\xi| + |\tau|^\frac12 < 2 ,\\
&\vspace{-2mm}\\
& \displaystyle\hat \phi(\xi,\tau)=0 \quad \mbox{ elsewhere }, \\
&\vspace{-2mm}\\
&\displaystyle\sum_{ -\infty < i < \infty } \hat \phi(2^{-i}\xi, 2^{-2i}\tau) =1 \ (
(\xi,\tau) \neq (0,0)).
\end{array}\right.
\end{eqnarray*}
We then introduce functions  $\phi_i  \in {\mathcal S}({\mathbb R}^{n+1})$, $i\in \mathbb Z$ in terms of $\phi$, which are defined by
\begin{eqnarray}\label{psi1}
\begin{array}{ll}
\widehat{\phi_i}(\xi, \tau) &= \hat \phi(2^{-i} \xi, 2^{-2i} \tau) ,\qquad i = 0, \pm 1, \pm 2 , \cdots.
\end{array}
\end{eqnarray}
Note that $\phi_i(x,t) = 2^{(i+2)n} \phi(2^i x, 2^{2i} t)$.
For $\al \in {\mathbb R}$ we define the anisotropic homogeneous Besov space
$\dot B^{\al,\frac12 \al}_{pr} ({\mathbb R}^{n+1})$ by
\begin{eqnarray*}
\dot B^{\al,\frac12 \al}_{pr} ({\mathbb R}^{n+1}) = \bket{ f \in {\mathcal
S}^{'}({\mathbb R}^{n+1}) \, | \, \|f\|_{\dot B^{\al, \frac12 \al}_{pr}} <
\infty \, }
\end{eqnarray*}
with the norms
\begin{align*}
 \|f\|_{\dot B^{\al, \frac12 \al}_{pr}} :&  =   \bke{
\sum_{ -\infty< i  < \infty} (2^{\al i} \|\phi_i *
f\|_{L^p})^r}^{\frac1r}, \quad 1 \leq r < \infty,\\
 \|f\|_{\dot B^{\al, \frac12 \al}_{p\infty}} :&  =
\sup_{ -\infty< i  < \infty} 2^{\al i} \|\phi_i *
f\|_{L^p}.
\end{align*}
Let $I$ be a open interval.
The anisotropic homogeneous Sobolev space $\dot W^{\al, \frac{\al}2}_p (\R_+ \times I)$ in $\R_+ \times I$ is defined by
\begin{align*}
\dot W^{\al, \frac{\al}2}_p (\R_+ \times I) = \bket{f= F |_{\R_+ \times I}  \, | \, F \in \dot W^{\al, \frac{\al}2}_p (\R \times {\mathbb R}) }
\end{align*}
with norm
\begin{align*}
\| f\|_{\dot W^{\al, \frac{\al}2}_p (\R_+ \times {\mathbb R}) } = \inf \bket{ \| F\|_{\dot W^{\al, \frac{\al}2}_p (\R \times {\mathbb R})  }  \, | \, F \in \dot W^{\al, \frac{\al}2}_p (\R \times {\mathbb R}), \,\, F|_{\R_+ \times I} =f }.
\end{align*}
Similarly, we define the the anisotropic homogeneous Besov space $\dot B^{\al, \frac{\al}2}_{pq} (\R_+ \times I)$.
The  properties of the anisotropic Besov spaces are comparable with the properties of  Besov spaces.
In particular, the following properties will be used later.

\begin{prop}
\label{prop2}
Let $\Om$ be $\R $ or $\R_+$.
\begin{itemize}
\item[(1)]
For $\al > 0$
\begin{align*}
\dot B^{\al,\frac{\al}{2}}_{pp} \bke{\Omega\times I}   &=L^p \bke{I;\dot B^\al_{pp}(\Omega)}\cap L^p \bke{\Omega; \dot B^{\frac{\al}{2}}_{pp}(I)}.
 \end{align*}

\item[(2)]
Suppose that $ 1 \leq p_0\leq p_1  \leq \infty, \,  1 \leq q_0\leq q_1  \leq \infty$ and $ \al_0\geq \al_1$ with $\al_0 - \frac{n+2}{p_0} = \al_1 - \frac{n+2}{p_1}$.
Then, the following inclusion holds
\begin{align*}
%&\dot W^{s_0,\frac{s_0}{2}}_{p_0} (\R \times {\mathbb R}) \subset  \dot B^{s_1,\frac{s_1}{2}}_{p_1} (\R \times {\mathbb R}), \qquad
%\dot  B^{s_0,\frac{s_0}{2}}_{p_0 q_0} (\Omega\times I) \subset   \dot W^{s_1,\frac{s_1}{2}}_{p_1 q_1} (\Omega\times I),\qquad
\dot  B^{\al_0,\frac{\al_0}{2}}_{p_0 q_0} (\Omega\times I) \subset   \dot B^{\al_1,\frac{\al_1}{2}}_{p_1 q_1} (\Omega\times I).
\end{align*}
%
%If $ 1 <  p_0\leq p_1 < \infty$ and $ s_0\geq s_1$ with $s_0 - \frac{n+2}{p_0} = s_1 - \frac{n+2}{p_1}$.
%Then, the following inclusion holds
%\begin{align*}
%%\label{prop3}
%\dot W^{s_0,\frac{s_0}{2}}_{p_0} (\R \times {\mathbb R}) \subset  \dot B^{s_1,\frac{s_1}{2}}_{p_1} (\R \times {\mathbb R}), \\
%%\mbox{ with }
%%\|f \|_{\dot H^{s_1,\frac{s_1}{2}}_{p_1} (\R \times {\mathbb R})   } \leq c \| f \|_{ \dot H^{s_0,\frac{s_0}{2}}_{p_0} (\R \times {\mathbb R})  },\\
%\dot B^{s_0,\frac{s_0}{2}}_{p_0} (\R \times {\mathbb R}) \subset  \dot W^{s_1,\frac{s_1}{2}}_{p_1 } (\R \times {\mathbb R}).
%%\mbox{ with }
%%\label{prop3}
%%\|f \|_{\dot B^{s_1,\frac{s_1}{2}}_{p_1 q_1 } (\R \times {\mathbb R})   } \leq c \| f \|_{ \dot B^{s_0,\frac{s_0}{2}}_{p_0 q_0} (\R \times {\mathbb R})  }.
%\end{align*}

\item[(3)]
%\footnote{KK:\color{red} Let $f \in  \dot W_p^{\al, \frac{\al}2} (\Omega\times I)$ or $f \in  \dot B_p^{\al, \frac{\al}2} (\Omega\times I)$ with $\al > \frac1p$. Then, $f|_{x_n =0} \in  \dot B_p^{\al -\frac1p, \frac{\al}2 -\frac1{2p}}(\Rn \times I)$ and
%\begin{align*}
%\| f \|_{ \dot B_p^{\al -\frac1p, \frac{\al}2 -\frac1{2p}} (\Rn \times I)} \leq c \| f \|_{ \dot W_p^{\al, \frac{\al}2} ( \Omega\times I)},\\
%\| f \|_{\dot  B_p^{\al -\frac1p, \frac{\al}2 -\frac1{2p}} (\Rn \times I)} \leq c \| f \|_{ \dot B_p^{\al, \frac{\al}2} ({\color{blue}\Omega} \times I)}.
%\end{align*} }
Suppose that $f \in  \dot W_p^{\al, \frac{\al}2} (\Omega\times I)$ and $f \in  \dot B_{pp}^{\al, \frac{\al}2} (\Omega\times I)$ with $\al > \frac1p$. Then,  $f|_{x_n =0} \in  \dot B_{pp}^{\al -\frac1p, \frac{\al}2 -\frac1{2p}}(\Rn \times I)$ and following estimates are satisfied.
\begin{align*}
\| f \|_{ \dot B_{pp}^{\al -\frac1p, \frac{\al}2 -\frac1{2p}} (\Rn \times I)} \leq c \| f \|_{ \dot W_p^{\al, \frac{\al}2} ( \Omega \times I)},\\
\| f \|_{\dot  B_{pp}^{\al -\frac1p, \frac{\al}2 -\frac1{2p}} (\Rn \times I)} \leq c \| f \|_{ \dot B_{pp}^{\al, \frac{\al}2} ( \Omega \times I)}.
\end{align*}

%\item[(5)]
%For $f \in  \dot W_p^{\al, \frac{\al}2} (\R \times {\mathbb R})$ and $f \in  \dot B_p^{\al, \frac{\al}2} (\R \times {\mathbb R}), \, \al > \frac2p$, $f|_{t =0 } \in  \dot B_p^{\al -\frac2p }(\R )$ with
%\begin{align*}
%\| f|_{t =0} \|_{ \dot B_p^{\al -\frac2p } (\R )} \leq c \| f \|_{ \dot W_p^{\al, \frac{\al}2} (\R \times {\mathbb R})},\quad
%\| f|_{t=0} \|_{ \dot B_p^{\al -\frac2p } (\R )} \leq c \| f \|_{ \dot B_p^{\al, \frac{\al}2} (\R \times {\mathbb R})}.
%\end{align*}
%
 \end{itemize}
\end{prop}

\begin{rem}
For the proof of $(1)$ in Proposition \ref{prop2}, one can refer to \cite[Theorem 3]{DT}.
We also consult Theorem 6.5.1 and Theorem 6.6.1 in \cite{BL} for $(2)$ and $(3)$  in Proposition \ref{prop2}, respectively.
\end{rem}

We next remind the non-homogeneous anisotropic Sobolev space and Besov space defined by
\begin{align*}
\| f \|_{W^{\al, \frac{\al}2}_p (\Om \times I)}& = \| f \|_{L^p (\Om \times I)} + \| f \|_{\dot W^{\al, \frac{\al}2}_p (\Om \times I)},\\
\| f \|_{B^{\al, \frac{\al}2}_{pp} (\Om \times I)}& = \| f \|_{L^p (\Om \times I)} + \| f \|_{\dot B^{\al, \frac{\al}2}_{pp} (\Om \times (0, I)},
\end{align*}
where $ \al > 0$ and $1 \leq p \leq \infty$.
Then, Proposition \ref{prop2} holds for non-homogeneous anisotropic Sobolev spaces and Besov spaces (see \cite{amman-anisotropic, Triebel2, Triebel}).

Let $\Om$ be a domain in $\R$ and $I$ be an open interval. Let $X(\Om)$ be a function space defined in $\Om$. We denote by  $L^q_t X_x (\Om \times I), \, 1 \leq q \leq \infty$  with norm
\[
\|f\|_{L^q_t X_x (\Om \times I)}=\Big(\int_I \|f(t)\|_{X(\Om)}^q dt\Big)^{\frac{1}{q}}.
\]
In case that $X(\Omega)=L^q(\Omega)$, we denote $L^q(\Om \times I) = L^q_t L^q_x(\Om \times I)$.

%For mixed norm space, we denote that
%\begin{align*}
%L^q_tW^{k,p}(\Om \times (0, T)): = L^q(0, T; W^{k,p}(\Om))
%\end{align*}

%\begin{rem}\label{rem0208}
%The properties in Proposition \ref{prop2} of the homogeneous anisotropic Besov spaces in $\R \times {\mathbb R}$ hold for the homogeneous Besov spaces in $\R_+ \times  {\mathbb R}_+$ and $\Rn \times  {\mathbb R}_+ $ (see \cite{amman-anisotropic, Triebel,Triebel2}).
%
%
%
%\end{rem}

%We shall consider the trace operator
%\begin{align*}
%Tr : {\mathcal S} ({\mathbb R}^{n+1}) \ri {\mathcal S} ({\mathbb R}^{n}), \quad
%(Tr f)(x',t)  =  f(x', 0,t).
%\end{align*}
%Now, we introduce the trace theorem of the homogeneous Sobolev space.
%\begin{prop}\label{prop0308}
%Let $1 < p < \infty$ and $\al > \frac1p$. Then, the trace operator can be extended  by
%\begin{align}\label{trace}
%Tr : \dot W^{\al, \frac{\al}2}_p (\R_+ \times {\mathbb R }) \ri \dot B^{\al -\frac1p, \frac{\al}2 - \frac1{2p}}_{pp} (\Rn \times {\mathbb R }),\\
%\label{trace2} Tr :  W^{\al, \frac{\al}2}_p (\R_+ \times {\mathbb R }) \ri  B^{\al -\frac1p, \frac{\al}2 - \frac1{2p}}_{pp} (\Rn \times {\mathbb R }).
%\end{align}
%\end{prop}
%%The trace theorem for generalized Sobolev spaces are found in Theorem 6.6.1 in \cite{BL}. With the same argument, Proposition \ref{prop0308} is proved.
%
%
%
%For the details of the proof of Proposition \ref{prop0308}, we refer
% Theorem 6.6.1  in \cite{BL}. As a matter
%of fact, in \cite{BL}, the usual nonhomogeneous Sobolev case has been dealt,
%however, the proofs found in \cite{BL} could be easily modified to
%the parabolic homogeneous Sobolev.

We introduce notions of {\it very weak solutions} for the
Stokes equations and the Navier-Stokes equations with non-zero boundary
values in a half-space $\hR$. To be more precise, we consider first
the following Stokes equations in $\hR$:
\begin{equation}\label{Stokes-bvp-200}
w_t - \De w + \na \pi = f, \qquad \mbox{div } w =0,\quad
\mbox{ in }
 \hR \times (0,\infty)
\end{equation}
with zero initial data and  non-zero boundary value
\begin{equation}\label{Stokes-bvp-210}
\rabs{w}_{t=0}= 0, \qquad  \rabs{w}_{x_n =0} = g.
\end{equation}
We now define  very weak solutions of the Stokes equations
\eqref{Stokes-bvp-200}-\eqref{Stokes-bvp-210}.
\begin{defin}\label{stokesdefinition}
Let $g \in L^1_{\rm{loc}} (\Rn \times (0, \infty))$ and $ f \in
L^1_{\rm{loc}}(\hR\times (0, \infty))$. A vector field $w\in
L^1_{\rm{loc}}( \hR \times (0, \infty))$ is called a very weak solution of the Stokes equations \eqref{Stokes-bvp-200}-\eqref{Stokes-bvp-210},
if the following equality is  satisfied:
\begin{align*}
-\int^\infty_0\int_{\hR }w \cdot \Delta \Phi
dxdt=\int^\infty_0\int_{\hR}\bke{w \cdot \Phi_t+ f \cdot \Phi} dxdt
-\int^\infty_0 \int_{\Rn} g \cdot D_{x_n}\Phi  dx' dt
\end{align*}
for each $\Phi\in C^2_c(\overline{\hR}\times [0,\infty))$ with
\begin{align}\label{0528-1}
\mbox{\rm div } \Phi=0,\quad \Phi\big|_{\Rn \times (0, \infty)}=0.
\end{align}
In addition, for each $\Psi\in C^1_c(\overline{\hR})$ with
 $\Psi \big|_{\Rn}=0$
\begin{equation}\label{Stokes-bvp-2200}
\int_{\hR} w(x,t) \cdot \na \Psi(x) dx =0  \quad  \mbox{ for all}
\quad 0 < t< \infty.
\end{equation}
\end{defin}

We recall some estimates which are useful to our purpose.
In the sequel, we denote by $N$ and $\Gamma$ the fundamental solutions of the Laplace and the heat equation, respectively, i.e.
\begin{align*}
N(x) = \left\{\begin{array}{ll}\vspace{2mm} \displaystyle -\frac{1}{(n-2)\om_n} \frac{1}{|x|^{n-2}},& n \geq 3\\
\displaystyle \frac1{2\pi} \ln |x|, & n =2
\end{array}
\right. ,\qquad
\Gamma(x,t)  = \left\{\begin{array}{ll} \vspace{2mm}
\displaystyle\frac{1}{(4\pi t)^{\frac{n}{2}}} e^{ -\frac{|x|^2}{4t} }, &   t > 0,\\
0, & t < 0,
\end{array}
\right.
\end{align*}
where $\om_n$ is the measure of the unit sphere in $\R$.
\begin{prop}\label{prop0803-1}
Let $w$ be solution of \eqref{Stokes-bvp-200}-\eqref{Stokes-bvp-210} with $g =0$.
Then,
\begin{align*}%\label{apendix1}
\| w(t) \|_{L^p (\R_+)} \leq c \big( \|\Ga * {\mathbb P} f(t) \|_{L^p (\R_+)}  +  \|\Ga^* * {\mathbb P} f(t) \|_{L^p (\R_+)} \big) \quad 0 < t < \infty,
\end{align*}
where ${\mathbb P}$ is Helmholtz decomposition in $\R_+$ and
\begin{align*}
\Ga * {\mathbb P} f(x,t) & = \int_0^t \int_{\R_+} \Ga(x -y, t -s) {\mathbb P}f(y,s) dyds,\\
\Ga^* *  {\mathbb P}f(x,t)& = \int_0^t \int_{\R_+} \Ga(x' -y', x_n + y_n , t -s) {\mathbb P}f(y,s) dyds.
\end{align*}
\begin{proof}
See the proof of Lemma 3.3 in \cite{CK}.
\end{proof}

\end{prop}

\begin{prop}\label{prop0803-2}
Let $f = {\rm div} \, F$ with $\rabs{F_{in}}_{x_n =0} =0$ and $g = 0$. Then, for $ 1 < p_0 \leq p< \infty$,
\begin{align*}%\label{apendix2}
 \|\Ga *  {\mathbb P} f(t) \|_{L^p (\R_+)}  +  \|\Ga^* * {\mathbb P } f(t) \|_{L^p (\R_+)}  \leq  c \int_0^t (t-s)^{-\frac12 -\frac{n}2 (\frac1{p_0} -\frac1{p})} \| F(s)\|_{L^{p_0}(\R_+)} ds. %\\
%\|\na \Ga *  {\mathbb P} f(t) \|_{L^p (\R_+)}  +  \|\na \Ga^* * {\mathbb P } f(t) \|_{L^p (\R_+)}  \leq  c \int_0^t (t-s)^{-1 -\frac{n}2 (\frac1{p_0} -\frac1{p})} \| F(s)\|_{L^{p_0}(\R_+)} ds.
\end{align*}
\end{prop}
\begin{proof}
See the proof of  Lemma 3.7 in \cite{CJ2}.
\end{proof}
%Applying the real interpolation for integral form, we have
%\begin{prop}\label{prop0803}
%
%\begin{align}\label{apendix2}
% \|\na \Ga *  {\mathbb P} \tilde f_1(t) \|_{\dot C^{\al }_r (\R_+)} & \leq c \|\na \Ga *  {\mathbb P} \tilde f_1(t) \|_{\dot B^{\al +\frac{n}r }_r (\R_+)}  %+  \| \na \Ga^* * {\mathbb P } f(t) \|_{\dot B^{\al + \frac{n}r }_r (\R_+)}  %\nonumber\\
% \leq  c \int_0^t (t -s)^{-\frac12 -\frac{\al}{2} -\frac{n}{2r}} \| \tilde f_1(s)\|_{L^r (\R_+)} ds,
%\end{align}
%\end{prop}

With aid of Proposition \ref{prop0803-1} and Proposition \ref{prop0803-2}, we obtain the following proposition.
\begin{prop}\label{proposition1}
Let $1 < p < \infty$.
Let $f = {\rm div} \, F$ with $\rabs{F_{in}}_{x_n =0} =0$ and $g= 0$. Then, for $1 < p < \infty$ the solution $w$ of \eqref{Stokes-bvp-200}-\eqref{Stokes-bvp-210}  satisfies
\begin{align*}
\| w(t)\|_{L^{p} (\R_+)} & \leq c \int_0^t ( t -s)^{-\frac12} \| F (t)\|_{L^p (\R_+)} ds, \qquad 0 < t < \infty.
%& \leq c  \|W \|_{L^\infty (\R_+ \times (0, 1))}\int_0^t ( t -s)^{-\frac12} \| W (s)\|_{L^p (\R_+)}  ds.
\end{align*}
\end{prop}
%\begin{proof}
%
%
%
%
%
%In the proof of Lemma 3.3 in \cite{CK}, the authors showed that $w$ defined by
%\eqref{rep-bvp-stokes-w} has the following estimate
%\begin{align}\label{apendix1}
%\| w(t) \|_{L^p (\R_+)} \leq c \big( \|\Ga * {\mathbb P} f(t) \|_{L^p (\R_+)}  +  \|\Ga^* * {\mathbb P} f(t) \|_{L^p (\R_+)} \big),
%\end{align}
%where
%\begin{align*}
%\Ga * {\mathbb P} f(x,t) & = \int_0^t \int_{\R_+} \Ga(x -y, t -s) {\mathbb P}f(y,s) dyds,\\
%\Ga^* *  {\mathbb P}f(x,t)& = \int_0^t \int_{\R_+} \Ga(x' -y', x_n + y_n , t -s) {\mathbb P}f(y,s) dyds.
%\end{align*}
%Let $f = {\rm div} \, F$ with $F_{in}|_{x_n =0} =0$.
%From the proof of  Lemma 3.7 in \cite{CJ2}, we have
%\begin{align}\label{apendix2}
% \|\Ga *  {\mathbb P} f(t) \|_{L^p (\R_+)}  +  \|\Ga^* * {\mathbb P } f(t) \|_{L^p (\R_+)}  \leq  c \int_0^t (t -s)^{-\frac12} \| F(s)\|_{L^p (\R_+)} ds.
%\end{align}
%From \eqref{apendix1} and \eqref{apendix2}, we complete the proof of Proposition \ref{proposition1}.
%\end{proof}
%The proof of Proposition \ref{proposition1} is given in Appendix \ref{appendix0131-1}.
%\footnote{KK: \color{red} We may just refer to the way of proof for Lemma 3.4 in [4]. We need to specify the range of $p$, i.e. $1\le p\le infty$.}

Some estimates of  solutions for \eqref{Stokes-bvp-200} are reminded as well.

\begin{prop}\label{theo0503}(\cite[Proposition 2.3]{CKca})
Let $f = {\rm div} \, F$ with $F \in L^q (0,\infty; L^p (\R_+))$, $\rabs{F}_{x_n =0} \in L^q (0, \infty; \dot B^{-\frac1p}_{pp} (\Rn))$ and  $g =0$. Then, for $1< p, q< \infty$ the solution $w$ of \eqref{Stokes-bvp-200}-\eqref{Stokes-bvp-210}  satisfies
\begin{align}\label{0504-1}
\|  \na w\|_{L^q (0, \infty;   L^p (\R_+))} \leq c \bke{\| F\|_{L^{q} (0,\infty;  L^p (\R_+))} +  \| \rabs{F}_{x_n =0}\|_{L^q (0, \infty; \dot B^{-\frac1p}_{pp}(\Rn))}}.
\end{align}
\end{prop}

\begin{prop}
\label{thm-stokes}\cite[Theorem 1.2]{CJ}
Let $g =0$ and $f=\mbox{\rm div } \, F$ with $F\in L^q(\R\times{\mathbb R}_+), \, 1<q<\infty $. Then   there is a unique very weak solution $w\in  L^q({\mathbb R}^n_+\times (0,T) ) $ of the Stokes equations
\eqref{Stokes-bvp-200}-\eqref{Stokes-bvp-210}
satisfying   the following inequality
\begin{align*}
\| w\|_{L^{q}({\mathbb R}^n_+\times (0,T))}
 &\leq cT^{\frac{1}{2}}\|F\|_{L^q(\R \times (0,T))}.
\end{align*}
\end{prop}

%In \cite{GGS}, Giga etal. showed \eqref{0504-1} with additional condition  ${\rm div}\, F =0$ and $F_{in}|_{x_n =0}=0, \, i =1, \cdots , n$. In \cite{KS}, it was shown  that  the second term of the right-hand side in \eqref{0504-1} is dropped with condition $p =q$.

%\footnote{KK: \color{red}The above sentence can be shifted after Proposition 2.4.}

\begin{prop}\label{theoexternel-boundary}\cite[Theorem 2.2]{CKsima}
%Let $\Om $ be a bounded domain with $C^{2,\al}$, $0<\al\leq 1$
%boundaries in $\R$, $n \geq 2$.
Let $ n +2 <  p < \infty $. Let $ f= {\rm div} \, F$ with  $ F
\in L^p (\R_+ \times (0, T)) $ and $ g \equiv 0$.   Then, there exists unique  very weak solution, $w \in L^\infty (\R_+ \times (0, T))$, of the Stokes equations
\eqref{Stokes-bvp-200} such that
\begin{equation*}%\label{externel-boundary}
\| w\|_{L^\infty (\R_+ \times (0, T))} \leq c T^{\frac12(1-\frac{n+2}{p} )} \max(1, T^{ -\frac{\al_1}2}) \|  F  \|_{ L^p (\R_+ \times (0, T))},
\end{equation*}
where  $0< \al_1 \le 1 -\frac{n+2}p$.
\end{prop}

\begin{prop}\label{theoexternel-boundary-2}
Let $\frac{n +2}2 <  p < \infty $. Let $ f \in L^p (\R_+ \times (0, T)) $ and $ g = 0$.   Then, there exists unique very weak solution, $w \in L^\infty (\R_+ \times (0, T))$, of the Stokes equations
\eqref{Stokes-bvp-200} such that
\begin{equation*}%\label{externel-boundary-2}
\| w\|_{L^\infty (\R_+ \times (0, T))} \leq c T^{1-\frac{n+2}{2p} } \|  f\|_{L^p (\R_+ \times (0, T))}.
\end{equation*}
\end{prop}
The proof of Proposition \ref{theoexternel-boundary-2} is rather straightforward, and thus we skip its details.

%We provide the proof of  Proposition \ref{theoexternel-boundary-2} in Appendix \ref{appendix0909}.%{\color{red}$\bullet$}\footnote{it may be better to skip proving the proposition.}
%\begin{prop}
%Let $1 < p < \infty$ and $\al > 1 +\frac1p$.
%Let ${\mathcal F} \in  L^q (0, \infty; H^\al_p(\R_+ ))$. Then, $p \in L^q (0, \infty; \dot H^{\al -1}_p (\R_+))$ with
%\begin{align*}
%\|p\|_{L^q (0, \infty; \dot H^{\al -1}_p (\R_+))} \| \leq  \| {\mathcal F}\|_{L^q (0, \infty; H^\al_p(\R_+ ))}.
%\end{align*}
%
%\end{prop}

We similalrly consider the
Navier-Stokes equations in a half-space, namely
\begin{equation}\label{NSE-bvp-230}
u_t - \De u + \na p =-\mbox{div}(u\otimes u), \qquad \mbox{div } u
=0 \qquad\mbox{ in } \quad  Q^+: = \R_+ \times (0, \infty)
\end{equation}
with  zero initial data and  non-zero boundary values
\begin{equation}\label{NSE-bvp-240}
\rabs{u}_{t=0}= 0, \qquad  \rabs{u}_{x_n =0} = g.
\end{equation}
We mean by  very weak solutions of the Navier-Stokes equations
\eqref{NSE-bvp-230}-\eqref{NSE-bvp-240} distribution solutions, which are defined as follows:

\begin{defin}\label{NS-vweak}
Let $g \in L^1_{\rm{loc}}(\pa \hR\times (0, \infty))$. A vector field
$u\in L^2_{loc}( \hR \times (0, \infty))$ is called a very weak solution
of the non-stationary Navier-Stokes equations
\eqref{NSE-bvp-230}-\eqref{NSE-bvp-240}, if the following equality
is  satisfied:
\begin{align*}%\label{NSE-bvp-250}
-\int^\infty_0\int_{\hR}u\cdot \Delta \Phi dxdt=\int^\infty_0\int_{\hR}(u\cdot
\Phi_t+(u\otimes u):\nabla \Phi) dxdt -\int^\infty_0 \int_{\pa \hR} g
\cdot D_{x_n} \Phi dx' dt
\end{align*}
for each $\Phi\in C^2_c(\overline{\hR}\times [0,\infty))$ satisfying \eqref{0528-1}.
In addition, for each $\Psi\in C^1_c(\overline{\hR})$ with
 $\Psi \big|_{\pa \hR }=0$,  $u$ satisfies \eqref{Stokes-bvp-2200}.
\end{defin}

%%%%%%%%%%%%%%%%%%%%%%%%%%%%%%%%%%%%%%%%%%%%%%%%%
%%%%%%%%%%%%%%%%%%%%%%%%%%%%%%%%%%%%%%%%%%%%%%%%%
%%%%%%%%%%%%%%%%%%%%%%%%%%%%%%%%%%%%%%%%%%%%%%%%%

\section{Stokes equations with boundary data in a half-space}
\label{SS-half}
\setcounter{equation}{0}

%{\color{magenta} In this section we provide the proof of Theorem \ref{maintheorem-SS} for Stokes equations.}

For convenience, we introduce a tensor $L_{ij}$ defined by
\begin{align}\label{L-tensor}
L_{ij} (x,t) & =  D_{x_j}\int_0^{x_n} \int_{\Rn}   D_{z_n}
\Ga(z,t) D_{x_i}   N(x-z)  dz,\quad
i,j=1,2,\cdots,n.%\\
%\notag A(x,t) & =\int_{\Rn}\Ga(z^{\prime},0,t)N(x^{\prime}-z^{\prime},x_n)dz^{\prime}.
\end{align}
We recall the following relations on $L_{ij}$  (see \cite{So}):
\begin{align} \label{1006-3}
\sum_{i=1}^{n} L_{ii} = \frac12 D_{x_n} \Ga, \quad\qquad  L_{in} =
L_{ni}  - B_{in} \,\, \mbox{ if }\, i \neq n,
\end{align}
where
\begin{equation}\label{B-tensor}
B_{in}(x,t) = \int_{\Rn}D_{x_n} \Ga(x^\prime -z^\prime ,
x_n, t) D_{z_i} N( z^{\prime},0) dz^\prime.
\end{equation}
Furthermore, we remind an estimate of $L_{ij}$ defined in
\eqref{L-tensor} (see \cite{So})
\begin{equation}\label{est-L-tensor}
|D^{l_0}_{x_n} D^{k_0}_{x'} D_{t}^{m_0} L_{ij}(x,t)| \leq
\frac{c}{t^{m_0 + \frac12} (|x|^2 +t )^{\frac12 n + \frac12 k_0}
(x_n^2 +t)^{\frac12 l_0}},
\end{equation}
where $ 1 \leq  i \leq n$ and $1 \leq j \leq n-1$.

It is known that the Poisson kernel $K $ of the Stokes equations is  given as follows   (see \cite{So}):
\begin{align}\label{Poisson-tensor-K}
\notag K_{ij}(x'-y',x_n,t) &  =  -2 \delta_{ij} D_{x_n}\Ga(x'-y', x_n,t)
+4L_{ij} (x'-y',x_n,t)\\
& \qquad +2  \de_{jn} \de(t)  D_{x_i} N(x'-y',x_n),
\end{align}
where $\delta(t)$ is the Dirac delta function and $\delta_{ij}$ is
the Kronecker delta function, and
thus, the solution $w$ of the Stokes
equations \eqref{Stokes-bvp-200}-\eqref{Stokes-bvp-210} with $ f =0$ is expressed by
\begin{align}\label{rep-bvp-stokes-w}
w_i(x,t) = \sum_{j=1}^{n}\int_0^t \int_{\Rn} K_{ij}(
x^{\prime}-y^{\prime},x_n,t-s)g_j(y^{\prime},s) dy^{\prime}ds.
\end{align}

Next, we will construct a solution $w$ of Stokes equations via
\eqref{rep-bvp-stokes-w}  for a certain $g$ such that $ w \in L^\infty$ but
$L^p$-norm of $\nabla w$ is not bounded near boundary.
 For convenience, we denote
\[
A = \{ y'=(y_1, y_2, \cdots, y_{n-1})\in \Rn \, | \, 3 < |y'| < 4\sqrt{n}, \,  -4\sqrt{n} < y_i < -3, \, \, 1 \leq i \leq n-1 \}.
\]
We introduce
a non-zero boundary data $g:\mathbb R^{n-1}\times \mathbb
R_+\rightarrow  \R$ with only $n-$th component defined
as follows:
\begin{align}\label{0502-6}
g(y', s) = (0, \cdots, 0,  g_n (y',s))= (0, \cdots, 0,a g^{\calS}_n (y')g^{\calT}_n(s)),
\end{align}
where $g_n^{\mathcal S} \geq 0$ and $g_n^{\mathcal T} \geq 0$ satisfy
%{\footnote{KK:{\color{blue} Notation convension: how about $g^1_n$ and $g^2_n$ as $g^s_n$ and $g^t_n$? }}}
\begin{align}\label{boundarydata}
g^{\mathcal S}_n \in C^\infty_c (A), \qquad
 {\rm supp} \,\, g_n^{\mathcal T} \subset (\frac{3}{4},\frac78), \quad   g^{\mathcal T}_n   \in L^\infty({\mathbb R}).
\end{align}
Here $a> 0$ is a constant, which is specified later. In this section, we assume $a=1$, without loss of generality (in section \ref{proof-ns}, the parameter $a$ will be taken sufficiently small).

In the next lemma, we estimate spatial derivatives of the convolution of $L_{ni}$ and $g_n$ up to the second order. For convenience, we denote
\begin{equation}\label{wi-Lgn}
w^{{\mathcal L}}_{i}(x,t) = \int_0^t \int_{\Rn} L_{ni}(
x^{\prime}-y^{\prime},x_n,t-s)g_n(y^{\prime},s) dy^{\prime}ds, \qquad 1 \leq i \leq n.
\end{equation}
\begin{lemm}\label{lemm0706-1}
%\footnote{KK: \color{red}$g_n$ must be specified in Lemma 3.1.}
Let $g_n$ be given in \eqref{0502-6} and \eqref{boundarydata}, and $w^{{\mathcal L}}_{i}$ defined in \eqref{wi-Lgn}.
%\begin{align*}
%w^{{\mathcal L}}_{i}(x,t) = \int_0^t \int_{\Rn} L_{ni}(
%x^{\prime}-y^{\prime},x_n,t-s)g_n(y^{\prime},s) dy^{\prime}ds \quad 1 \leq i \leq n.
%\end{align*}
Then, for $|x'| < 2$ and $ t > 0$, $w^{{\mathcal L}}_{i}$ satisfies the following estimates.
\begin{align}
\label{0706-6}|& w^{{\mathcal L}}_{i}(x,t) | \leq  c \| g\|_{L^\infty}, \qquad 1 \leq i \leq n,\\
\label{0706-7}& |\na w^{{\mathcal L}}_{i}(x,t) | \leq  c \| g\|_{L^\infty}, \qquad 1 \leq i \leq n-1,\\
\label{0706-7-1}& |\na w^{{\mathcal L}}_{n}(x,t) | \leq  c ( 1 +  (\ln \frac{x_n^2}t)_+ )\| g\|_{L^\infty}, \\
\label{0706-8} & |\na^2 w^{{\mathcal L}}_{i}(x,t) | \leq c ( 1 + (\ln \frac{x_n^2}t)_+ )\| g\|_{L^\infty}, \qquad 1 \leq i \leq n-1.
%\label{0706-8-1}& \|\na^2 w_n^{\mathcal L}\|_{L^q(0, \infty, L^{r}(\R_+))}
%\leq c  \| g^{\mathcal T}_n\|_{L^{\infty}} \| D_{x'} g^{\mathcal S}_n\|_{L^{\infty}}, \qquad \frac1q +1 >   1 -\frac1{2r}.
\end{align}

\end{lemm}

\begin{proof}
Let $1 \leq i \leq n-1$.
Noting that for $|x'| < 2$ and $y' \in A$, we have $|x' -y'| > 1$. Hence, due to the estimate \eqref{est-L-tensor}, we have that
\begin{align}
\label{0705-2} |   w^{{\mathcal L}}_{i} (x,t)|+|D_{x_k}   w^{{\mathcal L}}_{i} (x,t)|  & \leq c \| g\|_{L^\infty}, \quad 1 \leq k \leq n-1,\\
\label{0705-2-2}|D_{x_k} D_{x_l} w^{{\mathcal L}}_{i} (x,t)|
 & \leq c \| g\|_{L^\infty}, \quad  1 \leq k, l \leq n-1.
\end{align}
From \eqref{0705-2}, we get \eqref{0706-6} for $1 \leq i \leq n-1$.
%Hence, for $i \neq n$ or $j \neq n$, we have
%\begin{align}\label{ineqn}
%\| D_{x_i} D_{x_j} w^1 \|_{L^{r}} \leq  \| g\|_{L^\infty} \quad \forall \,  r < \infty.
%\end{align}

Set $f(x,t) = \int_0^t \int_{\Rn} D_{x_n} \Ga(x' -z', x_n, t-s) g_n(z',s) dz'ds$. Then, we note that   $ -B_{in} (x,t) = R'_i  f(x,t) $,  where $R'_i$ is $n -1$ dimensional Riesz transform. From the second identity of \eqref{1006-3},  we observe
\begin{align}\label{0819-1}
w^{{\mathcal L}}_i(x,t)
& =  D_{x_n} D_{x_i}\int_0^{x_n} \int_{\Rn} N(x-y) f(y,t)dy - R'_i  f(x,t).
\end{align}
On the other hand, we note that
\begin{align*}
\De_x \int_0^{x_n} \int_{\Rn} N(x-y) f(y,t)dy &  = \frac12 f(x,t)  + I( D_{x_n} f(\cdot, x_n,t))(x'),
\end{align*}
where  $I f (\cdot, x_n,t) = \int_{\Rn} N(x' - y', 0) f(y', x_n,t) dy'$. Since $D_{x_i} I = R'_i$, we have from \eqref{0819-1}
\begin{align}\label{0706-1}
\notag  D_{x_n} w^{{\mathcal L}}_i (x,t) & = - \De_{x'}   D_{x_i} \int_0^{x_n} \int_{\Rn} N(x-y) f(y,t)dy\\
\notag & \quad   +  D_{x_i} \big( \frac12 f(x,t) + I( D_{x_n} f(x', x_n,t))(x')  \big) -   D_{x_n}  R'_i  f(x,t)\\
& = - \sum_{k =1}^{n-1} D_{x_k}  w^{{\mathcal L}}_{ik} (x,t)      + \frac12   D_{x_i}    f(x,t),
\end{align}
where
\begin{align*}
w^{{\mathcal L}}_{ik} (x,t)   = \int_0^t \int_{\Rn} L_{ik} (x' -y', x_n, t-s) g_n(y',s) dy' ds,\quad k=1, 2, \cdots, n-1.
\end{align*}
It follows from \eqref{est-L-tensor} that, for $|x' | < 2$,
\begin{align}\label{0706-2}
| D_{x'} w^{{\mathcal L}}_{ik} (x,t) |+|  D_{x_i} f(x,t) | \leq  \| g_n \|_{L^\infty}.
\end{align}
Summing up \eqref{0705-2}, \eqref{0706-1} and \eqref{0706-2}, we obtain \eqref{0706-7}.

% we showed that for $1 \leq i \leq n-1$,
%\begin{align}\label{0706-5}
%| \na w^{{\mathcal L}}_i (x,t) | \leq  c\| g\|_{L^\infty}.
%\end{align}

Similarly, for $1 \leq l \leq n$, we have from \eqref{0706-1}
\begin{align}\label{0819-2}
  D_{x_n} D_{x_l} w^{{\mathcal L}}_i (x,t)
& = - \sum_{k =1}^{k = n-1}  D_{x_l} D_{x_k}  w^{{\mathcal L}}_{ik} (x,t)      + \frac12   D_{x_l} D_{x_i}    f(x,t).
\end{align}
Here, we note that
\begin{align}\label{0819-3}
|D_{x_l} D_{x_i}    f(x,t)|\leq c\| g\|_{L^\infty}.
\end{align}
Again, using the estimate \eqref{est-L-tensor}, for $1 \leq l \leq n$ and $1 \leq k \leq n-1$,  we get
\begin{align}\label{0819-4}
|D_{x_l} D_{x_k}  w^{{\mathcal L}}_{ik} (x,t)|  \leq c \int_0^t(t-s)^{-\frac12} ( x_n^2 + t-s)^{-\frac12} ds \| g\|_{L^\infty}
\leq c (1 +   (\ln \frac{x_n^2}t)_+ ) \| g\|_{L^\infty}.
\end{align}
Hence, combining estimates \eqref{0705-2-2}, \eqref{0819-2},  \eqref{0819-3} and \eqref{0819-4}, we obtain \eqref{0706-8}.

%we have
%\begin{align}\label{0706-4}
%| \na^2 w^{{\mathcal L}}_i (x,t) \leq c |\ln \frac{x_n^2}t|\| g\|_{L^\infty}.
%\end{align}

Now, it remains to estimate $w_n$. The first equality of  \eqref{1006-3} implies
\begin{align}\label{0819-6}
     w^{{\mathcal L}}_n (x,t)
& = - \sum_{k =1}^{n-1}   w^{{\mathcal L}}_{kk} (x,t)      -2      f(x,t).
\end{align}
Hence, we observe that,  for $1 \leq k \leq n$,
\begin{align}\label{0819-7}
|   w^{{\mathcal L}}_n (x,t)|
  \leq c \| g\|_{L^\infty}, \quad
|D_{x_k}   w^{{\mathcal L}}_{n} (x,t)|
  \leq c ( 1 +  (\ln \frac{x_n^2}t)_+ ) \| g\|_{L^\infty}.
\end{align}
%and, for $1 \leq l \leq n-1$,
%\begin{align}\label{0819-8}
%| D_{x_n} w^{{\mathcal L}}_n (x,t)|+| D_{x_n} D_{x_l} w^{{\mathcal L}}_n (x,t)| \leq c \abs{\ln \frac{x_n^2}t}\| g\|_{L^\infty}.
%\end{align}
Therefore, the estimates \eqref{0706-6} and \eqref{0706-7-1} are consequences of  \eqref{0819-7}. We complete the proof.
\end{proof}

Next lemma shows a pointwise control for $n -1$ dimensional Riesz transform
of $e^{-\frac{|x'|^2}t}$, which is one of crucial estimates in our analysis.

\begin{lemm}\label{lemma0709-1}
Let $ 1\leq |X'| \leq 5$. Then,
\begin{align*}
  \int_{\Rn}   e^{-\frac{|X' -z'|^2}{4t}} \frac{z_1}{| z'|^n} dz = c_{n-1} t^{\frac{n-1}2 }  \frac{X_1  }{|X' |^{n}}  + J(X', t),
\end{align*}
where
%\begin{align*}
$|J(X', t) | \leq  c  t^{\frac{n}2  }$ and $c_{n-1}=(4\pi)^{\frac{n-1}{2}}$.
%\end{align*}
\end{lemm}
\begin{proof}
We divide $\Rn$ by three disjoint sets
$D_1, D_2$ and $D_3$, which are defined by
\[
D_1=\bket{z'\in\Rn: |X'-z'| \leq \frac1{10} |X'|},
\]
\[
D_2=\bket{z'\in\Rn: |z'| \leq \frac1{10} |X'|}, \qquad D_3=\Rn\setminus
(D_1\cup D_2).
\]
We then split the following integral into three terms as follows:
\begin{align}\label{0730-2}
\int_{\Rn} e^{-\frac{| X'-z'|^2}{4t}}  \frac{z_1}{|z'|^n} dz' =
\int_{D_1}\cdots + \int_{D_2} \cdots+ \int_{D_3}\cdots := J_1 + J_2
+J_3.
\end{align}
Since $\int_{D_2} \frac{z_1}{|z'|^n} dz' =0$,  we have $\int_{D_2} \frac{z_1}{|z'|^n}  e^{-\frac{|X'-z'|^2}{4t}}   dz' = \int_{D_2} \frac{z_1}{|z'|^n} \big(
e^{-\frac{|X'-z'|^2}{4t}} - e^{-\frac{|X'|^2}{4t}} \big) dz'$.
%Since $ \frac12 \leq |X' | \leq 2$,
Thus, using the mean-value Theorem, we  have
\begin{align}\label{est-J2}
\notag |J_2| & =\abs{\int_{D_2}   \frac{z_1}{|z'|^n} \big(
 e^{-\frac{|X'-z'|^2}{4t}} -  e^{-\frac{|X'|^2}{4t}} \big) dz'
}\leq   ct^{-1  }|X'| e^{-c\frac{|X'|^2}{t}}  \int_{D_2}
\frac{1}{|z'|^{n-2}}   dz'\\
&\leq  c t^{-1  }|X'|^2 e^{-c\frac{|X'|^2}{4t}} \leq  c
e^{-c\frac{|X'|^2}{4t}}.
\end{align}
Since $\int_{|z'| > a} e^{-|z'|^2} dz' \leq c_1 e^{-c_2 a^2}, \, a > 0$, we have
\begin{equation}\label{est-J3}
 |J_3| \leq  \frac{c}{|X'|^{n-1}  }   \int_{\{|z'-X'| \geq \frac1{10} |X'|\}} e^{-\frac{|z'-X'|^2}{4t}} dz'\leq   \frac{ct^{\frac{n-1}{2} }}{|X'|^{n-1}}  e^{-c\frac{ |X'|^2}{t}}\leq   c
e^{-c \frac{|X'|^2}{t}}.
\end{equation}
%\begin{align}
%\notag |J_3| &   \leq  \frac{c}{|X'|^{n-1}  }   \int_{\{|z'-X'| \geq \frac1{10} |X'|\}} e^{-\frac{|z'-X'|^2}{t}} dz'\\
%\notag & =   \frac{ct^{\frac{n-1}2 }}{|X'|^{n-1}} \int_{\{|z'| \geq \frac1{10}\frac{ |X'|}{\sqrt{t}}\}} e^{- |z'|^2} dz'\\
%\label{est-J3}
%&\leq   \frac{ct^{\frac{n-1}{2} }}{|X'|^{n-1}}  e^{-c\frac{ |X'|^2}{t}}\leq   c
%e^{-c \frac{|X'|^2}{t}}.
%\end{align}
Due to    $ 1 \leq |X'| \leq 5$, it follows  from \eqref{est-J2} and \eqref{est-J3} that
\begin{align}\label{0604}
|J_2 (X',t)|+|J_3(X',t)| \leq   c e^{-\frac{c}{ t}}.
\end{align}

%For $|x'| < \frac12$ and $y'
%\in A$,  we note that
%$X_1 = x_1 - y_1 \geq -\frac12 + 1 =\frac12$ and
%$\frac15 |X'| \leq \frac15 (|x'| + | y'|) \leq \frac15 \cdot \frac52
%\leq X_1$. Then, for $|X'-z'| \leq \frac1{10} |X'| $,  we have
%\[
%z_1 = z_1 -X_1 + X_1 \geq X_1 - |z_1 - X_1| \geq X_1 - |X' -z'|
%\]
%\[
%\geq X_1 - \frac1{10} |X'| \geq \frac15|X'| - \frac1{10} |X'| = \frac1{10} |X'|.
%\]

Now, we estimate $J_1$. Firstly, we decompose $J_1$ in the following way:
%Therefore,  we obtain
\begin{align}\label{est-J1}
\notag J_1 & = \int_{D_1}   e^{-\frac{ |X' -z'|^2}{4t} } \frac{z_1}{|z'|^n} dz'
%\\
%\notag &
 =  (4t)^{\frac{n-1}2   } \int_{\{|z'| \leq  \frac1{10}\frac{ |X'|}{\sqrt{t}}\}}    e^{-|z'|^2 } \frac{X_1 - 2t^\frac12 z_1}{|X' - 2t^\frac12 z'|^{n}}dz'
\\
\notag &
 =  (4t)^{\frac{n-1}2 } \int_{\{|z'| \leq  \frac1{10}\frac{ |X'|}{\sqrt{t}}\}}     e^{-|z'|^2 } \Big( \frac{X_1 - 2t^\frac12 z_1}{|X' - 2t^\frac12 z'|^{n}} - \frac{X_1  }{|X' |^{n}} \Big)dz'\\
\notag &\quad - (4t)^{\frac{n-1}2  } \frac{X_1  }{|X' |^{n}} \int_{\{|z'| \geq  \frac1{10}\frac{ |X'|}{\sqrt{t}}\}}   e^{-|z'|^2 }   dz'
%\\
%\notag & \quad
+ (4t)^{\frac{n-1}2 } \frac{X_1  }{|X' |^{n}} \int_{\Rn}    e^{-|z'|^2 }   dz'\\
& = J_{11} + J_{12} + J_{13}.
\end{align}
We observe that
\begin{align}\label{0730-1}
|J_{11} (X',t)| \leq ct^{\frac{n}2  }, \qquad |J_{12} (X',t)| \leq   e^{-c\frac{|X'|^2}{t}} \leq ct^{\frac{n}2 }.
\end{align}
Here we set $c_{n-1}:=\int_{\Rn}    e^{-|z'|^2 }  dz'=\pi^{\frac{n-1}{2}}$
and take $J: = J_2 + J_3 + J_{11} + J_{12}$. Then, combining \eqref{0730-2},  \eqref{0604}, \eqref{est-J1} and \eqref{0730-1}, we complete the proof of Lemma \ref{lemma0709-1}.
\end{proof}

%%%%%%%%%%%%%%%%%%%%%%%%%%%%%%%%%%%%%%%%%%%%%%%%%
%%%%%%%%%%%%%%%%%%%%%%%%%%%%%%%%%%%%%%%%%%%%%%%%%
%%%%%%%%%%%%%%%%%%%%%%%%%%%%%%%%%%%%%%%%%%%%%%%%%

\section{Proof of Theorem \ref{maintheorem-SS} for Stokes equations}
\label{thm1-1ss}
\setcounter{equation}{0}

The case for the Stokes system in Theorem \ref{maintheorem-SS} can be verified by the following proposition, where a class of boundary data for temporal variable is specified to show that velocity is bounded but its gradient is not integrable in $L^q_{\rm loc}$ near boundary. The case of the Navier-Stokes equations will be treated in subsection \ref{ns-thm11}.

%We are ready to prove Theorem \ref{maintheorem-SS} for Stokes equations.
\begin{prop}\label{lemma0406}
Let $1 < p < \infty$ and $g$  satisfy \eqref{0502-6} and \eqref{boundarydata}. Assume further that $g_n^{{\mathcal T}} \in L^\infty({\mathbb R}) \setminus \dot B^{\frac12 -\frac1{2p}}_{pp} ({\mathbb R})$.
Suppose that $w$ is a solution of the Stokes equations
\eqref{Stokes-bvp-200}-\eqref{Stokes-bvp-210} defined by
\eqref{rep-bvp-stokes-w} with $f=0$  and  the boundary data $g$.  Then, $w$ is bounded in $B_2^+ \times (0, 4)$ but the normal derivative of tangential components for $w$ are unbounded in $L^p(Q^+_{\frac12}) $, i.e. $w$
satisfies
\begin{equation}\label{L4Lp-w}
\| w\|_{L^\infty(B_2^+ \times (0, 4)) }<\infty,
\end{equation}
\begin{equation}\label{L2-grad-w}
\int^{1}_{\frac34} \int_{B^+_\frac12} |D_{x_n} w_i(x,t) |^{p} dxdt =
\infty, \qquad  i =1,\cdots, n-1.
\end{equation}
\end{prop}
\begin{rem}\label{rem0314}
In Appendix \ref{appendix0131}, we give an example, for a clearer understanding, of a function $ g^{\mathcal T}_n \in L^\infty({\mathbb R}) \setminus \dot B^{\frac12 -\frac1{2p}}_{pp} ({\mathbb R})$, $1<p<\infty$.
\end{rem}
\begin{proof}
We prove only the case that $i =1$, since the arguments are similar for other cases.
%{\footnote{KK: {\color{magenta} May we refer to previous results for $L^{\infty}$ estimates?}}}
From  \eqref{1006-3} and \eqref{Poisson-tensor-K},  we have
%\begin{align*}
%K_{1n}(x'-y',x_n,t) & = -L_{1n} (x'-y',x_n,t) +     \de(t)  D_{x_1} N(x'-y',x_n)\\
%& = -L_{n1} (x'-y',x_n,t) -B_{n1}(x'-y', x_n, t) +     \de(t)  D_{x_1} N(x'-y',x_n).
%\end{align*}
%{\footnote{KK: {\color{magenta} Notational convention: how about $w^1_{1}, w^2_{1}, w^3_{1}$ as $w^L_1, w^B_1, w^N_1$?}}}
\begin{align}\label{1220-8}
\notag w_1(x,t) & =    4\int_0^t   \int_{A} L_{n1}(x'- y', x_n, t
-s) g_n(y', s) dy' ds\\
\notag & \quad  + 4\int_0^t   \int_{A} B_{1 n}(x'- y', x_n, t
-s) g_n(y', s) dy' ds \\
\notag &\quad
-   2\int_{A} D_{x_1} N(x' -y', x_n)g_n(y',t) dy'\\
&:=   w^{\calL}_{1}(x,t) + w^{\calB}_{1}(x,t)+ w^{\calN}_1(x,t).
\end{align}
We note that $ 1 \leq |x' -y'| \leq 4\sqrt{n} $ for $|x'| \leq 2  $ and $
y' \in A$, and thus,  for $(x, t) \in B_2^+ \times (0,4) $ it is direct that
\begin{align}\label{0706-9}
\|w^{\calN}_1(x,t)\|_{L^\infty(B_2^+ \times (0,4))}+\|\na w^{\calN}_1(x,t)\|_{L^p(B_2^+ \times (0,4))} \leq c\| g_n\|_{L^\infty} \quad 1 < p < \infty.
\end{align}
%On the other hand, using the estimate \eqref{est-L-tensor}, we can see that
%\begin{align*}
%|w^{\calL}_1(x,t)|\le c\| g_n\|_{L^\infty},  \qquad  |\na w^{\calL}_1(x,t)| \leq c  | \ln \frac{x_n^2}t| \| g_n\|_{L^\infty}.
%\end{align*}
From Lemma \ref{lemm0706-1}, it is straightforward that for $1 < r < \infty$,
\begin{align}\label{0411-w}
\| w^{\calL}_1\|_{L^\infty(B_2^+ \times (0,4))}+\| \na w^{\calL}_1\|_{L^r(B_2^+ \times (0,4))} \leq c\| g_n\|_{L^\infty}.
%\| g_n\|_{L^\infty (\Rn \times (0, \infty))}.
\end{align}

%Since $L_{1n} = L_{n1} + B_{1n}$ by the second equality of \eqref{1006-3}, we divide $w_1^1= w_1^{11} + w_1^{12}$ by
%\begin{align}\label{0411-v}
%\notag w^1_1(x,t) &=  \int_0^t \int_{\Rn} L_{n1}(x'- y', x_n, t -s) g_n(y',
%s)dy'ds\\
%\notag & \qquad  +\int_0^t \int_{\Rn} B_{1n}(x'- y', x_n, t -s) g_n(y', s)dy'ds\\
%&: =
%w^{11}_1 (x,t) + w^{12}_1 (x,t).
%\end{align}
%
%
%
%
%From similar estimate with $w_1$, we have
%\begin{align*}
%\| w^{11}_1\|_{L^\infty(Q^+_1)}, \quad  \| \na w^{11}_1\|_{L^{p}(Q^+_1)} \leq \| g_n\|_{L^\infty (\Rn \times (0, \infty))}.
%\end{align*}
%
%
%
%
%
%
%

Next, we estimate $w_1^{\calB}$. Since $1 \leq |x' -y'| \leq 4\sqrt{n}$,  reminding \eqref{B-tensor} and Lemma \ref{lemma0709-1}, we note that
\begin{align}\label{0710-3}
 \notag  w^{\calB}_1(x,t) &=   c_n \int_{0}^t \int_{\Rn} g_n(y', s)   \frac{x_n}{(t -s)^{\frac{n+2}2}}e^{-\frac{x_n^2}{4(t-s)}}
\int_{\Rn} e^{-\frac{|x'-y'-z'|^2}{4(t-s)}}    \frac{z_1}{|z'|^n}   dz' dy' ds\\
& : = w^{\calB,1}_1(x,t) + w^{\calB,2}_1(x,t),
\end{align}
where
\begin{align}
w^{\calB,1}_1(x,t)& = c_n \int_{0}^t    g^{\mathcal T}_n( s)  \frac{x_n}{(t -s)^{\frac{3}2}}e^{-\frac{x_n^2}{4(t-s)}} \int_{\Rn}    g^{\mathcal S}_n(y')      \frac{x_1 -y_1  }{|x' -y' |^{n-1}}
 dy' ds,\\
w^{\calB,2}_1(x,t)& = c_n \int_{0}^t    g^{\mathcal T}_n( s)  \frac{x_n}{(t -s)^{\frac{n+2}2}}e^{-\frac{x_n^2}{4(t-s)}} \int_{\Rn}    g^{\mathcal S}_n(y')  J(x' -y', t-s)
 dy' ds.
\end{align}
Since $1 \leq |x' -y'| \leq 5$, it follows from Lemma \ref{lemma0709-1} that
\begin{align} \label{w12-1}
  | w^{\calB}_1(x,t)| &\leq    c_n \int_{0}^t  e^{-\frac{x_n^2}{t-s}} \big( \frac{x_n}{(t -s)^{\frac{3}2}}   + \frac{x_n}{t -s}  \big) ds \| g_n\|_{L^\infty}
  \leq    c_n  \| g_n\|_{L^\infty}.
\end{align}
Thus, due to \eqref{1220-8}, \eqref{0706-9} and \eqref{w12-1}, we obtain \eqref{L4Lp-w}.
Noting that
\begin{align*}
 D_{x'}   w^{\calB}_1(x,t) &=   c_n \int_{0}^t \int_{\Rn} D_{y'} g_n(y', s) \frac{x_n}{(t -s)^{\frac{n+2}2}}e^{-\frac{x_n^2}{t-s}}
\int_{\Rn} e^{-\frac{|x'-y'-z'|^2}{t-s}}    \frac{z_1}{|z'|^n}   dz' dy' ds,
\end{align*}
we similarly have
\begin{align} \label{w12-2}
  |D_{x'} w^{\calB}_1(x,t)| &\leq    c_n  \| D_{y'} g_n\|_{L^\infty}.
\end{align}
Using Lemma \ref{lemma0709-1}, we have
\begin{align}\label{0710-2}
 \notag  D_{x_n} w^{\calB,2}_1(x,t) &=     c_n \int_{0}^t \int_{\Rn}  g_n(y', s) \frac{1}{(t -s)^{\frac{n+2}2}}e^{-\frac{x_n^2}{4(t-s)}} ( 1 -\frac{x_n^2}{t-s} )
J(x' -y', t-s) ds\\
\notag & \leq c  \int_{0}^t     \frac{1}{t -s }e^{-\frac{x_n^2}{4(t-s)}} ( 1 -\frac{x_n^2}{t-s} )
  ds\| g_n\|_{L^\infty}\\
& \leq c ( 1 +|\ln \frac{x_n^2}t|)  \| g_n\|_{L^\infty}.
\end{align}
On the other hand, we
note that
\begin{align}\label{0710-1}
   w^{\calB,1 }_1(x,t)
&=   c_n \int_{0}^t D_{x_n} \Ga_1 (x_n, t-s) g^{{\mathcal T}}_n(s)ds \psi(x'),
\end{align}
where $\Ga_1$ is one dimensional Gaussian kernel and  $\psi(x') = \int_{\Rn} \frac{x_1 -y_1}{|x' -y'|^n} g_n^{{\mathcal S}}(y') dy'   $, which is smooth in $|x' | \leq 2$. Since ${\rm supp} \, g^{\mathcal T}_{n} \subset (\frac34, \frac78)$, we have for $t > 1$
\begin{align*}
 \abs{\int_{0}^t D^2_{x_n} \Ga_1 (x_n, t-s) g^{{\mathcal T}}_n(s)ds} & \leq      t^{-\frac32}   e^{-\frac{x_n^2}{t}} \| g^{{\mathcal T}}_n \|_{L^\infty}
\end{align*}
and for $x_n > 1$,
%\footnote{KK:\color{red} Are working on $Q_1^+$? If so, we don't need to consider $x_n>1$.}
we have
\begin{align*}
 |\int_{0}^t D^2_{x_n} \Ga_1 (x_n, t-s) g^{{\mathcal T}}_n(s)ds|
 \leq c \| g^{{\mathcal T}}_n \|_{L^\infty}
 \left\{\begin{array}{ll}
 x_n^{-1},\quad& t < 1,\\
   t^{-\frac32} e^{-\frac{x_n^2}t}, \quad& t > 1.
 \end{array}
 \right.
\end{align*}
This implies that
\begin{align}\label{0730-3-1}
\notag \int_1^\infty  \int_0^\infty       |\int_{0}^t D^2_{x_n} \Ga_1 (x_n, t-s) g^{{\mathcal T}}_n(s)ds|^p           dx_n dt
& \leq  c\| g_n^{\mathcal T}\|_{L^\infty ({\mathbb R})},\\
\int_{1}^\infty  \int_0^\infty      |\int_{0}^t D^2_{x_n} \Ga_1 (x_n, t-s) g^{{\mathcal T}}_n(s)ds|^p          dt     dx_n & \leq c\| g_n^{\mathcal T}\|_{L^\infty ({\mathbb R})}.
\end{align}
By the trace theorem of anisotropic space  (see  (3) of Proposition \ref{prop2}), we have
\begin{align}\label{0730-4}
\notag \| g_n^{{\mathcal T}}\|_{\dot B_{pp}^{ \frac12 -\frac1{2p}} ({\mathbb R})} & \leq c \|\int_{0}^t D_{x_n} \Ga_1 (x_n, t-s) g^{{\mathcal T}}_n(s)ds\|_{\dot H^{1, \frac12}_{p} ({\mathbb R}_+ \times {\mathbb R})}\\
& \leq c \|D_{x_n}\int_{0}^t D_{x_n} \Ga_1 (x_n, t-s) g^{{\mathcal T}}_n(s)ds\|_{L^p ({\mathbb R}_+ \times {\mathbb R})}.
\end{align}
With the aid of \eqref{0710-1}, \eqref{0730-3-1} and \eqref{0730-4}, we conclude that
\begin{align*}
\|D_{x_n} w_1^{\calB}\|_{L^{p} (Q^+_\frac12)} & \geq c \|D_{x_n}\int_{0}^t  D_{x_n} \Ga_1 (x_n, t-s) g^{{\mathcal T}}_n(s)ds\|_{L^p ({\mathbb R}_+ \times {\mathbb R})} \| \psi\|_{L^p (B'_1)} \\
&\quad -c \| g_n\|_{L^\infty (\Rn \times (0, \infty))}\\
&\geq c \| g_n^{{\mathcal T}}\|_{\dot B_{pp}^{ \frac12 -\frac1{2p}} ({\mathbb R})} \| \psi\|_{L^p (B'_1)} - c \| g_n\|_{L^\infty (\Rn \times (0, \infty))}.
%\\
%& =\infty.
\end{align*}
Since the righthand side is unbounded, we obtain \eqref{L2-grad-w}.
Hence, we complete the proof  Proposition   \ref{lemma0406}.
\end{proof}

\section{Proof of Theorem \ref{Stokes-maximal}}
\label{thm1-4ss}
\setcounter{equation}{0}

Before we prove Theorem \ref{Stokes-maximal} for the Stokes system,
we first show an elementary estimate, which is useful for our purpose.
\begin{lemm}\label{lemma0803-1}
For $1  < p  <  \infty$, $1 \leq  r  \leq  \infty$ and $0 < \be <1$,
\begin{align*}%\label{apendix2}
\|\na \Ga *   f(t) \|_{\dot B^{\be }_{pr} (\R_+)} +  \|\na \Ga^* *    f(t) \|_{\dot B^{\be }_{pr} (\R_+)}
 %+  \| \na \Ga^* * {\mathbb P } f(t) \|_{\dot B^{\al + \frac{n}r }_r (\R_+)}  %\nonumber\\
 \leq  c \int_0^t (t -s)^{-\frac12 -\frac{\be}{2}  } \| f(s)\|_{L^p (\R_+)} ds,
\end{align*}
where
\begin{align*}
\Ga *   f(x,t) & = \int_0^t \int_{\R_+} \Ga(x -y, t -s) f(y,s) dyds,\\
\Ga^* *   f(x,t)& = \int_0^t \int_{\R_+} \Ga(x' -y', x_n + y_n , t -s)  f(y,s) dyds.
\end{align*}
\end{lemm}
\begin{proof}
We prove only the case $\Ga *   f$, since the other case can be treated similarly. For $h \in L^p (\R)$, we have
%Because the proofs are exactly same, we only prove the case $\Ga *   f$.  For $h \in L^p (\R)$, we have
\begin{align*}
\|  \na^k  \Ga_t * h \|_{L^p (\R)} \leq  ct^{-\frac{k}2} \| h \|_{L^p (\R)} \quad 0 < t < \infty,
\end{align*}
where $k$ is a non-negative integer and
\begin{align*}
\Ga_t *   h(x) & = \int_{\R} \Ga(x -y, t) h(y) dy.
\end{align*}
Using the property of real interpolation between $k =1$ and $k =2$, we have
\begin{align*}
\|  \na  \Ga_t * h \|_{\dot B^{\be}_{pr} (\R)} \leq  ct^{-\frac{1}2 -\frac{\be}2} \| h \|_{L^p (\R)}.
\end{align*}
Let $\tilde f(t)$ be a zero extension of $f(t)$ over $\R$. Then, we have
\begin{align*}
\|\na \Ga *   f(t) \|_{\dot B^{\be }_{pr} (\R_+)} \leq \|\na \Ga *   \tilde f(t) \|_{\dot B^{\be }_{pr} (\R)}
 & \leq  c \int_0^t \|\na \Ga_{t-s} *   \tilde f(s) \|_{\dot B^{\be }_{pr} (\R)} ds\\
&  \leq  c \int_0^t (t -s)^{-\frac12 -\frac{\be}{2}  } \| \tilde f(s)\|_{L^p (\R)} ds\\
&  \leq  c \int_0^t (t -s)^{-\frac12 -\frac{\be}{2}  } \|  f(s)\|_{L^p (\R_+)} ds.
\end{align*}
Therefore, we complete the proof of Lemma \ref{lemma0803-1}.
%
%From the direct calculations, we have
%\begin{align*}
%\|\na^k \Ga *   f(t)\|_{L^p(\R_+)} + \|\na^k \Ga^* *   f(t)\|_{L^p(\R_+)} & \leq  \int_0^t  (t -s)^{-\frac{k}2} \|f(s)\|_{L^p (\R_+)} ds, \quad k \geq 0.
%\end{align*}
%Applying Lemma C.1 in \cite{CJ2}, we obtain Lemma \ref{lemma0803-1}.
\end{proof}

%{\color{red}{We note that the estimate
%\eqref{SS-max-10} is followed by similar proof of
%\cite{Seregin00}. In addition, it is also direct that solutions are H\"{o}lder continuous.
%Since its verification is tedious repetition of the argument in \cite[Proposition 2]{Seregin00}, we skip its details. Thus, it remains to show the estimate \eqref{SS-max-20}, that is, derivative is also H\"{o}lder continuous.}}
To prove Theorem \ref{Stokes-maximal}, we change the local problem into a problem in a half-space by multiplying a test function. Using Bogoski's formula to control non-divergence free term caused localization, we appropriately  decompose the solution to compute estimates of H\"{o}lder continuity.

Firstly,
let $\phi_1 \in C_c^\infty({\mathbb R}^{n})$ be a  cut-off function satisfying $ \phi_1 \geq 0$,  ${\rm supp} \, \phi_1 \subset B_\frac1{\sqrt{2}}$ and $\phi_1 \equiv 1 $ in $B_{\frac38}$. Also, let  and $\phi_2 \in C_c^\infty(-\infty, \infty)$ be a  cut-off function satisfying $ \phi_2 \geq 0$,  ${\rm supp} \, \phi_2 \subset (\frac12, 2)$ and $\phi_2 \equiv 1 $ in $(\frac3{4}, 1)$. Let $\phi (x,t) =\phi_1 (x) \phi_2 (t)$.
Let $U = u\phi$ and $\Pi = \pi \phi$ such that $U|_{Q^+_\frac14} =u$ and $\Pi|_{Q^+_\frac14} =\pi$. Then, $(U, \Pi)$ satisfies the following equations;
\[
U_t -\De U +\na \Pi = \tilde f, \quad {\rm div } \, U =h \qquad
\mbox{in}\,\, \R_+\times (0, 1),\]
\[
U\big|_{t =0} =0, \qquad U\big|_{x_n =0} =0,
\]
%\footnote{KK: Why is initial data zero?. Is the test fucntion time-dependent as well?}
where
\begin{align*}
\tilde f = -2 (\na u) \na \phi  - \De \phi u +\phi_t u + \pi \na \phi, \qquad
h = \na \phi \cdot u.
\end{align*}
We note that   $\tilde f, \,  h, \nabla  h, \, h_t \in L^q_tL^r_x(\R_+ \times (0, 1)), \, 1 < r < \infty$ with
\begin{align}\label{080202}
\notag\| \tilde f\|_{L^q_tL^r_x(\R_+ \times (0, 1)) } & \leq  c \big( \|  \na u \|_{L^q_tL^r_x (Q^+_\frac1{\sqrt{2}})} +  \| u \|_{L^q_tL^r_x(  Q^+_\frac1{\sqrt{2}})}   + \| \pi \|_{L^q_tL^r_x(Q^+_\frac1{\sqrt{2}}) }\big),\\
\notag \| h\|_{L^q_tL^r_x(\R_+ \times (0, 1)) } & \leq  c  \| u \|_{L^q_tL^r_x(Q^+_\frac1{\sqrt{2}}) },\\
\notag \| \na h \|_{L^q_tL^r_x(\R_+ \times (0, 1)) } & \leq  c  \big( \| u \|_{L^q_tL^r_x(Q^+_\frac1{\sqrt{2}}) } + \| \na u \|_{L^q_tL^r_x(Q^+_\frac1{\sqrt{2}}) } \big),\\
\notag \| h_t \|_{L^q_tL^r_x(\R_+ \times (0, 1))  } & \leq  c  \big( \| u \|_{L^q_tL^r_x(Q^+_\frac1{\sqrt{2}}) } + \| u_t \|_{L^q_tL^r_x(Q^+_\frac1{\sqrt{2}}) } \big)\\
  &  \leq  c  \big( \| u \|_{L^q_tL^r_x(Q^+_1) } + \| \na u \|_{L^q_tL^r_x(Q^+_1) }  + \| \pi \|_{L^q_tL^r_x(Q^+_1) }\big).
\end{align}
For the fifth inequality, we used Proposition 1 in \cite{Seregin00}.

%We recall that $G$ is represented by the following integral form
%\begin{align*}
%G(x,t) =  \int_{B_{1/2} \setminus B_{3/8}} N(x,y) g(y,t) dy,
%\end{align*}
%where $N=\partial_n E(x-y)+\partial_n E(x^*-y)$.
%\footnote{KK: Is it correct?}

Let $H (x,t)= \int_{\R_+} E(x, y) h(y,t) dy$ be a Bogoski's formula (see \cite{galdi}) such that
%\footnote{Here domain is half-space but iit seems later domain  is a ball}
\[
{\rm div } \, H(\cdot,t) = h(\cdot,t),\quad  \mbox{in}\,\, \R_+ \qquad \mbox{and}\qquad H (\cdot, t)\big|_{x_n =0} =0 \quad \mbox{on}\,\, \bke{x_n =0}.
\]
Since $h(t) \in W_0^{1,p}(\R_+)$ for all $  0 < t< \infty$, $H(t)$ satisfies
\begin{align}\label{0822-1}
%\norm{\na G}_{L^q_tL^r_x} & \leq c\norm{g}_{L^q_tL^r_x},\\
\notag\| \na^k H (t)\|_{L^r (\R_+)} & \leq c \| \na^{k -1} h (t) \|_{L^r (\R_+)} \leq c \| \na^{k -1} h (t) \|_{L^r (Q^+_{\frac1{\sqrt{2}}})}, \quad k= 1,\,2 \quad  1 < r < \infty,\\
\| H_t (t)   \|_{ L^{r^*} (\R_+)} & \leq c  \|  \nabla H_t (t)\|_{L^r (\R_+)}  \leq c \| h_t (t) \|_{ L^r (\R_+)} \leq c \| h_t (t) \|_{ L^r ( Q^+_{\frac1{\sqrt{2}}})}, %\leq c \| g_t\|_{L^q_t L^{r}_x}.
\end{align}
where $r^* = \frac{nr}{n-r}$ for $r< n$ (See Chapter 3 in \cite{galdi}). 
Take $ 1 <r < \infty$  and $0 < \ep$ such that $ \al +\frac{n}r +\ep  < 1$. Using the Besov imbedding and the property of real interpolation in \eqref{0822-1},  we obtain
%\footnote{KK: Is Besov space defined in the half-space? Why is $r$ and $\epsilon$ taken as such a way?  how about
%\[
%\| \na  G(t) \|_{ C^\al(\R_+)} \leq c \| G(t) \|_{ B^{1 + \al }_r (\R_+)} \leq c \| g(t) \|_{\dot B^{  \al  }_r (B_1)}  \leq c \| g(t) \|_{ C^{  \al } (B_1)}?
%\]
%}
\begin{align}\label{0802-2}
\| \na  H \|_{ L^\infty_t\dot C_x^\al(\R_+ \times (0, 1))} \leq c \| \na  H \|_{  L^\infty_t \dot  B^{ \al + \frac{n}r }_{rrx} (\R_+ \times (0, 1))} \leq c \| h \|_{ L^\infty_t\dot B^{  \al + \frac{n}r }_{rrx} (Q^+_\frac1{\sqrt{2}})}  \leq c \| h \|_{  L^\infty_t C^{  \al + \frac{n}r +\ep  }_x (Q^+_\frac1{\sqrt{2}})}.
\end{align}
%{\color{red}{
%Hence, we have
%\begin{align}\label{0802-2}
%\| \na  H \|_{L^\infty(0, 1; C^\al(\R_+)) }   \leq c \| h \|_{L^\infty (0,1; C^{  \al + \frac{n}r +\ep  } (B_1))}.
%\end{align}
%}}
Note that ${\rm supp} \, h \subset B_\frac1{\sqrt{2}} \setminus B_\frac38$.
Since $|E(x,y)| \leq c$ for $|x| <\frac14$ and $|y|> \frac38$, in case that $|x| < \frac14$, we observe that
\begin{align*}
|\na H(x,t) - \na H(x,s)|& \leq   \int_{B_{\frac1{\sqrt{2}}} \setminus B_{\frac38}} |E(x,y)  (h(y,t) - h(y,s))| dy\\
& \leq  \| h\|_{ \dot C_t^{\frac{\al}2}L^\infty_x(Q_\frac1{\sqrt{2}}^+)} |t -s|^{\frac{\al}2}.
\end{align*}
%Using the Sobolev inequality, for $r > n$, $\ep > 0$ with $\al > \frac{n}r +\ep $, we have
%\begin{align*}
%|\na G(x,t) - \na G(x,s)| \leq \|\na^2 ( G(t) - G(s)) \|_{L^r (\R_+)} \leq \|\na ( g(t) - g(s)) \|_{L^r (B_1)} \leq c \| g\|_{L^\infty(B_1;  C^{\frac{\al}2} (0, 1))}  |t -s|^{\frac{\al}2}.
%\end{align*}
Hence, we have
\begin{align*}
\| \na  H\|_{\dot C_t^{\frac{\al}2}L^\infty_x(Q_\frac14^+)} \leq  c\| h\|_{\dot C_t^{\frac{\al}2}L^\infty_x(Q_\frac1{\sqrt{2}}^+)}.
\end{align*}
Since $\al + \frac{n}r +\ep < 1$, from Proposition 2 and Lemma 1 in \cite{Seregin00}, there are $ 1 < q_0<2$ and $ 1< p_0 < p$ such that
\begin{align*}
  \| h \|_{L^\infty_t C_x^{  \al + \frac{n}r +\ep  } (Q^+_\frac12)}
  &\leq c\| u \|_{L^\infty_t C_x^{  \al + \frac{n}r +\ep  } (Q^+_\frac12)}\\
   &  \leq c \big( \| \na u \|_{L^{q_0}_tL^{p_0}_x (Q^+_\frac34)}  + \|  u \|_{L^{q_0}_tL^{p_0}_x (Q^+_\frac34) } + \| \pi \|_{L^{q_0}_tL^{p_0}_x(Q^+_\frac34) }     \big)\\
 &  \leq c \big( \| \na u \|_{L^{q}_tL^{p}_x (Q^+_1) }  + \|  u \|_{L^{q}_tL^{p}_x (Q^+_1)} + \| \pi \|_{L^{q}_tL^{p}_x (Q^+_1) }     \big).
\end{align*}
%{\color{red}{
%Since $\al + \frac{n}r +\ep < 1$, from proposition 2 and Lemma 1 in \cite{Seregin00}, there are $ 1 < q_0, \, p_0 < 2$ such that
%\begin{align*}
% \| h \|_{L^\infty (0,1; C^{  \al + \frac{n}r +\ep  } (B_1))} + \| h\|_{L^\infty(B_1^+;  C^{\frac{\al}2} (0, 1))}
%  &\leq c\| u \|_{L^\infty (0, 1; C^{  \al + \frac{n}r +\ep  } (B_1))} + \| u\|_{L^\infty(B_1^+;  C^{\frac{\al}2} (0, 1))}\\
%   &  \leq c \big( \| \na u \|_{L^{q_0}L^{p_0} }  + \|  u \|_{L^{q_0}L^{p_0} } + \| \pi \|_{L^{q_0}L^{p_0} }     \big)\\
% &  \leq c \big( \| \na u \|_{L^{q}L^{p} }  + \|  u \|_{L^{q}L^{p} } + \| \pi \|_{L^{q}L^{p} }     \big).
%\end{align*}
%}}
Summing up all estimates, we have
\begin{align}\label{timeregular}
\| \na  H \|_{  C^{\frac{\al}2}_t C^{\al}_x(Q_\frac14^+ )} \leq c \big( \| \na u \|_{L^{q}_tL^{p}_x (Q_1^+ )  }  + \|  u \|_{L^{q}_tL^{p}_x(Q_1^+ ) } + \| \pi \|_{L^{q}_tL^{p}_x (Q_1^+ ) }     \big).
\end{align}
We then decompose $U=H+W$  in $Q^+_{\frac14}$ such that $W$ solves
the following equations:
\[
W_t -\De W +\na \Pi = \tilde f_1, \quad {\rm div } \, W =0, \quad
\mbox{in}\quad Q_+,
\]
\[
W\big|_{t =0} =0, \qquad W|_{x_n =0} =0,
\]
where $Q_+ : = \R_+ \times (0, 1)$ and  $\tilde f_1 = \tilde f - H_t + \De H$.
If $n \leq p< r$, then  choose $p_1 < n$ such that $r = \frac{np_1}{n -p_1}$. Then, from \eqref{080202}, \eqref{0822-1}, Sobolev imbedding and  Proposition 1 in \cite{Seregin00}, we have
\begin{align}\label{0802-1}
\notag \| \tilde f_1 \|_{L^q_t   L^{r }_x (Q_+) }
& \leq c \big( \|  \na u \|_{L^q_tL^{r }_x(Q^+_{\frac1{\sqrt{2}}}) } +  \| u \|_{L^q_tL^r_x(Q^+_{\frac1{\sqrt{2}}}) }   + \| \pi \|_{L^q_tL^{r }_x(Q^+_{\frac1{\sqrt{2}}}) } + \| u_t \|_{L^q_tL^{p_1}_x(Q^+_{\frac1{\sqrt{2}}}) } \big)\\
\notag  &\leq c \big( \|  \na^2 u \|_{L^q_tL^{p_1}_x(Q^+_{\frac1{\sqrt{2}}}) }  +  \| \na u \|_{L^q_tL^{p_1}_x (Q^+_{\frac1{\sqrt{2}}})}     + \|  \na \pi \|_{L^q_tL^{p_1}_x (Q^+_{\frac1{\sqrt{2}}})} + \|   \pi \|_{L^q_tL^{p_1}_x(Q^+_{\frac1{\sqrt{2}}}) } \big)\\
\notag  &\leq c \big(  \| \na u \|_{L^q_tL^{p_1}_x (Q^+_1)}      + \|   \pi \|_{L^q_tL^{p_1}_x(Q^+_1) } + \| u \|_{L^q_tL^{p_1}_x (Q^+_1) } \big)\\
  &\leq c \big(  \| \na u \|_{L^q_tL^{p}_x (Q^+_1)}      + \|   \pi \|_{L^q_tL^{p}_x(Q^+_1) } + \| u \|_{L^q_tL^{p}_x (Q^+_1) } \big).
\end{align}
From well-known result (see Theorem 1.1 in \cite{So1}), we have
\begin{align*}
\| \na^2 W\|_{L^q_t   L^{r }_x(Q_+) } + \| D_t W\|_{L^q_t   L^{r }_x (Q_+ )} + \| \na \Pi \|_{L^q_t   L^{r }_x(Q_+)} &  \leq c \| \tilde f_1\|_{L^q_t   L^{r }_x(Q^+_{\frac1{\sqrt{2}}})}.
%&\leq c \big(  \| \na u \|_{L^q_tL^p_x(Q^+_1) }      + \|   \pi \|_{L^q_tL^p_x(Q^+_1) } + \| u \|_{L^q_tL^p_x(Q^+_1) } \big).
\end{align*}
Then, due to \eqref{080202}, \eqref{0822-1} and \eqref{0802-1}, we have
\begin{align*}
&  \| \na^2 u\|_{L^q_t L^{r }_x( Q_\frac14^+ )} + \| D_t u\|_{L^q_t L^{r }_x( Q_\frac14^+ )} + \| \na \pi \|_{L^q_t   L^{r }_x( Q_\frac14^+ )}\\
& \quad  = \| \na^2 U\|_{L^q_t L^{r }_x( Q_\frac14^+ )} + \| D_t U\|_{L^q_t L^{r }_x( Q_\frac14^+ )} + \| \na \Pi \|_{L^q_t   L^{r }_x( Q_\frac14^+ )}\\
%&\quad  = \| \na^2 W\|_{L^q_t   L^{p^* }_x( Q_\frac12^+ ) } + \| D_t W\|_{L^q_t   L^{p^* }_x ( Q_\frac12^+ )} + \| \na \Pi \|_{L^q_t   L^{p^* }_x( Q_\frac12^+ )}\\
& \quad \leq  c ( \| \tilde f_1 \|_{L^q_t L^{r }_x(Q_+)} + \|\na^2  H\|_{L^q_tL^{r }_x(Q^+_{\frac1{\sqrt{2}}})} +  \| H_t \|_{L^q_tL^{r }_x(Q^+_{\frac1{\sqrt{2}}}) } )\\
& \quad \leq  c ( \| \tilde f_1 \|_{L^q_t L^{r }_x(Q_+)} + \|\na  h\|_{L^q_tL^{p_1 }_x(Q^+_{\frac1{\sqrt{2}}})} +  \| h_t \|_{L^q_tL^{p_1 }_x(Q^+_{\frac1{\sqrt{2}}}) } )\\
& \quad \leq  c \big(  \| \na u \|_{L^q_tL^p_x ( Q_1^+ )}      + \|   \pi \|_{L^q_tL^p_x( Q_1^+ ) } + \| u \|_{L^q_tL^p_x( Q_1^+ ) } \big).
\end{align*}
Hence, we obtain  \eqref{SS-max-10} for $p \geq n$.

In case that $p < n$, we set $p^* = \frac{np}{n -p}$.  If $p^* > n$ then for the same reason as the case $p \geq n$,  we have
\begin{align*}
  &\| \na^2 u\|_{L^q_t L^{p^* }_x( Q_\frac14^+ )} + \| D_t u\|_{L^q_t L^{p^* }_x( Q_\frac14^+ )} + \| \na \pi \|_{L^q_t   L^{p^* }_x( Q_\frac14^+ )}
 \\
 \leq & c \big(  \| \na u \|_{L^q_tL^p_x ( Q_1^+ )}      + \|   \pi \|_{L^q_tL^p_x( Q_1^+ ) } + \| u \|_{L^q_tL^p_x( Q_1^+ ) } \big).
\end{align*}
If $p^* < n$, then iterating the  process up to $p^* > n$, we obtain  \eqref{SS-max-10}.

Using Proposition \ref{prop0803-1} and the property of real interpolation, for $  1 < r_1 < \infty, \,\,  1 \leq  r_2 \leq \infty$ and $\al > 0$,  $W$  holds the following estimate
\begin{align}\label{apendix1}
\|  W(t) \|_{B^{\al}_{r_1r_2} (\R_+)} \leq c \big( \|\Ga * {\mathbb P} \tilde f_1(t) \|_{B^{\al}_{r_1r_2} (\R_+)}  +  \|\Ga^* * {\mathbb P} \tilde f_1(t) \|_{B^\al_{r_1r_2} (\R_+)} \big).
\end{align}
%where
%\begin{align*}
%\Ga * {\mathbb P} f(x,t) & = \int_0^t \int_{\R_+} \Ga(x -y, t -s) {\mathbb P}f(y,s) dyds,\\
%\Ga^* *  {\mathbb P}f(x,t)& = \int_0^t \int_{\R_+} \Ga(x' -y', x_n + y_n , t -s) {\mathbb P}f(y,s) dyds.
%\end{align*}
From the Besov imbedding, \eqref{apendix1} and  Lemma \ref{lemma0803-1}, we have
\begin{align}\label{apendix2}
 \|\na W(t) \|_{\dot C^{\al } (\R_+)}   & \leq c \|\na W(t) \|_{\dot B^{\al +\frac{n}r }_{rr} (\R_+)}  %+  \| \na \Ga^* * {\mathbb P } f(t) \|_{\dot B^{\al + \frac{n}r }_r (\R_+)}  %\nonumber\\
 \leq  c \int_0^t (t -s)^{-\frac12 -\frac{\al}{2} -\frac{n}{2r}} \| \tilde f_1(s)\|_{L^r (\R_+)} ds,
\end{align}
\begin{align}\label{apendix2-2}
 \notag\|\na W(t) \|_{L^\infty (\R_+)}
  \leq  \|\na W(t) \|_{\dot B^{\frac{n}r}_{r1} (\R_+)}
\notag & \leq  c \int_0^t (t -s)^{-\frac12  -\frac{n}{2r}} \| \tilde f_1(s)\|_{L^r (\R_+)} ds\\
 &\leq  c \int_0^t (t -s)^{-\frac12 -\frac{\al}{2} -\frac{n}{2r}} \| \tilde f_1(s)\|_{L^r (\R_+)} ds.
\end{align}
Combining \eqref{0802-1},  \eqref{apendix1}, \eqref{apendix2} and \eqref{apendix2-2} and H\"{o}lder inequality, and taking $ r < \infty$  satisfying $\frac{n}r +\frac2q   < 1 -\al$, we obtain
\begin{align}\label{0524-4}
 \| \na W\|_{L^\infty_t C_x^\al(Q_+)} &\leq c\| \tilde f_1 \|_{L^q_t L^r_x(Q_+)}
\leq  c \big(  \| \na u \|_{L^q_tL^p_x(Q^+_1) }      + \|   \pi \|_{L^q_tL^p_x (Q^+_1) } + \| u \|_{L^q_tL^p_x (Q^+_1) } \big).
\end{align}
Next, we compute H\"{o}lder continuous estimate with respect to $t$. From \cite{So1},
%{\color{red}$\bullet$}\footnote{it may be better to give fhe formula for $K^{{\mathcal P}}$.}
%{\color{red}$\bullet$}\footnote{We may use the estimate of unrestrictive Green tensor.}
there is the kernel $K^{{\mathcal P}}$ such that  $W$ is represented by
\begin{align}\label{0909-1}
W(x,t) = \int_0^t \int_{\R_+} K^{{\mathcal P}}(x,y, t-s) {\mathbb P} \tilde f_1(y,s) dyds,
\end{align}
where
\begin{align*}
K^{{\mathcal P}}_{ij}(x,y, t) &=\de_{ij} \big(\Ga(x-y,t) - \Ga(x -y^*,t) \big)\\
& \qquad  + 4(1 -\de_{jn}) D_{x_j} \int_0^{x_n} \int_{\Rn}\Ga(x -y^* -z, t) D_{z_i} N(z) dz.
\end{align*}
Furthermore, it is known that  $K^{{\mathcal P}}$  satisfies that for all $ k \in {\mathbb N} \cup \{0\}$, $l= (l', l_n) \in ({\mathbb N} \cup \{0\})^n$, (see Proposition 2.5 in \cite{So1})
\begin{align}\label{0801-1}
|D_t^k D_{x_n}^{l_n}D_{x'}^{l'} K^{{\mathcal P}}(x,y,t)| \leq \frac{ce^{-c_1 \frac{ y_n^2}t}}{ t^k (t + x_n^2)^{\frac{l_n}2} ( |x -y^*|^2 + t )^{\frac{n + |l'| }2}  }.
\end{align}
We remark that one can also refer \cite{KLLT21-G} for the representation formula via unrestricted Green tensor. Since
we use $L^p$-type estimate of $f$, not the pointwise estimate of $f$, the formula \eqref{0909-1} with the restricted Green tensor $K^{{\mathcal P}}$ is enough for our purpose.
Continuing computations, it follows that
\begin{align}\label{0707-1}
\notag\na_x W(x,t) -\na_x W(x,s)& = \int_0^s \int_{\R_+} (\na_x K^{{\mathcal P}}(x,y,t-\tau) -\na_x K^{{\mathcal P}}(x,y, s-\tau)) {\mathbb P} \tilde f_1(y,\tau) dyd\tau\\
\notag & \quad + \int_s^t \int_{\R_+} \na_xK^{{\mathcal P}}(x,y, t-\tau) {\mathbb P} \tilde f_1(y,\tau) dyd\tau\\
&: = I_1 + I_2.
\end{align}
Using H\"{o}lder inequality, we first estimate $I_2$
\begin{align*}
|I_2| \leq \int_s^t (\int_{\R_+} |\na_xK^{{\mathcal P}}(x,y, t-\tau)|^{r'}dy)^{\frac1{r'}}  \|{\mathbb P} \tilde f_1(\tau)\|_{L^r(\R_+)} d\tau.
\end{align*}
It follows from \eqref{0801-1} that
\begin{align*}
\int_{\R_+} |\na_xK^{{\mathcal P}}(x,y, t-\tau)|^{r'}dy
\leq &  c\int_{\R_+} (|x -y^*|^2 + (t -\tau))^{-\frac{n r'}2} ( x_n ^2 + t-\tau)^{-\frac{r'}2}  e^{-c_1\frac{y_n^2}{t -\tau}}  dy\\
 \leq   & c \int_0^\infty   ( x_n^2 + t-\tau)^{-\frac{r'}2 - \frac{nr'}2 +\frac{n-1}2 } e^{-c_1\frac{y_n^2}{t -\tau}}   dy_n\\
    = &     (   t-r)^{- \frac{r'}2  - \frac{nr'}2 +\frac{n}2 } .
\end{align*}
Hence, for $1 < r< \infty$ satisfying $1 > \frac{n}r + \frac2q$, we have
\begin{align}\label{0707-2}
\notag |I_2|& \leq c\int_s^t (t -\tau)^{-\frac12 -\frac{n}{2r}} \|\tilde f_1(\tau)\|_{L^r(\R_+)} d\tau\\
\notag & \leq  c(\int_s^t (t -\tau)^{(-\frac12 -\frac{n}{2r})q'} d \tau )^{\frac1{q'} }  \|\tilde f_1\|_{L^q_tL^r_x(Q_+) }\\
&  =    c(t -s)^{\frac12 -\frac{n}{2r} -\frac1q}   \|\tilde f_1\|_{L^q_tL^r_x(Q_+) }.
\end{align}

Next, we  estimate   $I_1$. Using H\"{o}lder inequality, it follows that
\begin{align*}
I_1 &= c\int_0^s \int_0^{t -s}  \int_{\R_+}   D_t \na_x K^{{\mathcal P}}(x,y,\eta + s-\tau)   {\mathbb P} \tilde f_1(y,\tau) dy  d \eta d\tau\\
& \leq c\int_0^s \int_0^{t-s} (\int_{\R_+} | D_t \na_x K^{{\mathcal P}}(x,y,\eta + s-\tau)|^{r'}dy)^{\frac1{r'}}  \|\tilde f_1(\tau)\|_{L^r(\R_+)} d\eta d\tau.
\end{align*}
Due to \eqref{0801-1}, we have
\begin{align*}
&\int_{\R_+} | D_t \na_x K^{{\mathcal P}}(x,y,\eta + s-\tau)|^{r'}dy \\
\leq   &c\int_{\R_+} \frac{e^{-c_1\frac{y_n^2}{\eta + s -\tau}}}{ (\eta + s-\tau)^{ r' }(|x -y|^2 + \eta + s-r)^{\frac{nr'}2} ( x_n ^2 + \eta + s-\tau)^\frac{r'}2}   dy\\
\leq   &c\int_0^\infty \frac{e^{-c_1\frac{y_n^2}{\eta + s -\tau}}}{ (\eta + s-\tau)^{ r' } ( x_n^2 + \eta + s-\tau)^{\frac{r'}2 + \frac{nr'}2 -\frac{n-1}2}}   dy_n\\
\leq  &   \frac{c}{  (\eta + s-\tau)^{ (\frac32 +\frac{n}{2r})r' }}.
\end{align*}
Using change of variable, we have
\begin{align*}
\int_0^{t-s} (\eta + s-\tau)^{-\frac32 -\frac{n}{2r}}  d\eta
& = (s -\tau)^{-\frac12 -\frac{n}{2r}} \int_0^{\frac{t-s}{s-\tau}} (\eta + 1)^{-\frac32 -\frac{n}{2r}}  d\eta\\
& \leq \left\{\begin{array}{lc}
c(t-s)(s-\tau)^{-\frac12 -\frac{n}{2r}} \quad &\mbox{if}\,\,\, \tau < 2s -t,\\
c(s-\tau)^{-\frac12 -\frac{n}{2r}} \quad &\mbox{if}\,\,\, \tau > 2s -t.
\end{array}
\right.
\end{align*}
We note that  for $1 > \frac{n}r + \frac2q$
\begin{align*}
(t-s) (\int_0^{2s-t}  ( s-\tau)^{(-\frac12 -\frac{n}{2r})q'}   d\tau)^{\frac1{q'}}
 \leq c(t-s)^{\frac32 -\frac{n}{2r} -\frac1q},\\
 (\int_{2s-t}^s  ( s-\tau)^{(-\frac12 -\frac{n}{2r})q'}   d\tau)^{\frac1{q'}}
 \leq c(t-s)^{\frac12 -\frac{n}{2r} -\frac1q}.
\end{align*}
Hence, we have
\begin{align}\label{0524-3}
|I_1|& \leq c   (t -s)^{\frac12 -\frac{n}{2r} -\frac2q}   \|\tilde f_1\|_{L^q_tL^r_x(Q_+)}.
\end{align}
Due to \eqref{0707-1}, \eqref{0707-2} and \eqref{0524-3}, for $\al \leq \frac12 -\frac{n}{2r} -\frac2q$, it follows that
\begin{align}\label{0524-5}
\| W\|_{L^\infty_x \dot C_t^{\frac{\al}2} (Q^+_\frac14)} \leq  \|\tilde f_1\|_{L^q_tL^r_x(Q_+)}.
\end{align}
Combining \eqref{0524-4} and \eqref{0524-5}, we obtain
\begin{align}\label{spaceregular}
\| W\|_{ C_t^{ \frac{\al}2} C_x^{\al} ( Q^+_\frac14 )} \leq  \|\tilde f_1\|_{L^q_tL^r_x(Q_+)} \leq  c \big(  \| \na u \|_{L^q_tL^p_x(Q^+_1) }      + \|   \pi \|_{L^q_tL^p_x(Q^+_1) } + \| u \|_{L^q_tL^p_x (Q^+_1)} \big).
\end{align}
Since $ U = W + H$ and $u = U$ in $Q^+_\frac14$, we have via \eqref{timeregular} and \eqref{spaceregular}
\begin{align*}
&\| \na u\|_{C^{\frac{\al}2}_t C^{\al}_x (Q^+_\frac14)} = \| \na U\|_{ C^{\frac{\al}2}_t C^{\al}_x (Q^+_\frac14)}\\
&\leq  c ( \| \tilde f_1 \|_{L^q_t L^r_x(Q_+)} + \| h \|_{L^\infty_t C^{  \al + \frac{n}r  }_x(Q^+_\frac1{\sqrt{2}})} +  \| h\|_{C_t^{\frac{\al}2}L^\infty(Q^+_\frac1{\sqrt{2}})})\\
& \leq  c \big(  \| \na u \|_{L^q_tL^p_x(Q_1) }      + \|   \pi \|_{L^q_tL^p_x(Q^+_1) } + \| u \|_{L^q_tL^p_x(Q^+_1) } \big) + c\big(\| u \|_{L^\infty_t C^{  \al + \frac{n}r  }_x (Q^+_\frac1{\sqrt{2}})} +  \| u\|_{   C^{\frac{\al}2}_t L^\infty_x (Q^+_\frac1{\sqrt{2}})}  \big).
\end{align*}
Due to Proposition 2 and Lemma 1 in \cite{Seregin00}, there are $1 < q_0<2$ and $ 1 <  p_0 < p$ such that
\begin{align*}
\| u \|_{L^\infty_t C_x^{  \al + \frac{n}r}(Q^+_\frac14)  } +  \| u\|_{C_t^{\frac{\al}2}L_t^\infty(Q^+_\frac14)  }
& \leq c \big(  \| \na u \|_{L^{q_0}_tL^{p_0}_x (Q^+_1)}      + \|   \pi \|_{L^{q_0}_tL^{p_0}_x(Q^+_1) } + \| u \|_{L^{q_0}_tL^{p_0}_x(Q^+_1) } \big)\\
& \leq c \big(  \| \na u \|_{L^{q}_tL^{p}_x(Q^+_1) }      + \|   \pi \|_{L^{q}_tL^{p}_x(Q^+_1) } + \| u \|_{L^{q}_tL^{p}_x (Q^+_1)} \big).
\end{align*}
Hence, we obtain \eqref{SS-max-20}. This completes the proof of Theorem \ref{Stokes-maximal}.
\qed

%%%%%%%%%%%%%%%%%%%%%%%%%%%%%%%%%%%%%%%%%%%%%%%%%
%%%%%%%%%%%%%%%%%%%%%%%%%%%%%%%%%%%%%%%%%%%%%%%%%
%%%%%%%%%%%%%%%%%%%%%%%%%%%%%%%%%%%%%%%%%%%%%%%%%

\section{Proof of Theorem \ref{maintheorempressure-SS} for Stokes equations}
\label{proofss}
\setcounter{equation}{0}

Let $g$ be a vector field defined in \eqref{0502-6} such that
\begin{align}\label{1222-2}
\notag g_n^{\calS} \in C_c^\infty (B'_2 \setminus B'_\frac32) , \quad    g_n^{\calT} \in C(0, 1) \quad \mbox{ with} \quad  g_n^{\calT}(0) =0, \quad  D_t g_n^{\calT} \in L^q (0, 1),\\
   D_t g_n^{\calT} \notin L^{r} (\frac34, 1) \quad \mbox{ for all} \quad  r>q, \quad g_n^{\calT} \notin B^{1 -\frac1{2r}}_{rr} (\frac34,\frac78) \quad \mbox{ for all} \quad  r>\frac{3q}2.
\end{align}
\begin{rem}\label{example-0924}
An example of $g_n^{\calT}$ satisfying \eqref{1222-2} is the following:
\begin{align*}
 g_n^{\calT} (t) = \left\{\begin{array}{l}
 0, \quad 0 < t \leq  \frac34, \vspace{2mm}\\
   (t -\frac3{4})^{1 -\frac1q} \ln (t-\frac3{4})^{-1} \quad \frac3{4} < t < \frac78.
 \end{array}
 \right.
\end{align*}
%where $\eta \in C^\infty_c(0, \frac32)$ satisfying $ 0 \leq \eta \leq 1$ and $\eta= 1$ in $(\frac3{4}, \frac32)$.
In Appendix \ref{appendix0131}, for clarity, we give its details.
\end{rem}

Let $\phi (x,t) = \displaystyle c_n \int_{\Rn} \frac1{|x-y'|^{n-2}} g_n(y',t) dy'  $. We define  $w^{{\mathcal H}} = \na \phi$ and $p^{{\mathcal H}}(x,t) =  - D_t \phi (x,t)$. Then, we note that $(w^{{\mathcal H}},p^{{\mathcal H}})$ is the solution of the Stokes equations in $\R_+ \times (0, \infty)$
\begin{align}\label{boundaryg}
\left\{\begin{array}{l}
w^{{\mathcal H}}_t -\De w^{{\mathcal H}} + \na p^{{\mathcal H}} =0, \qquad
%(x,t) \in \R_+ \times (0, \infty),\\
{\rm div }\,\, w^{{\mathcal H}} =0,\\
\vspace{-2mm}\\
w^{{\mathcal H}}|_{t =0} =0, \qquad w^{{\mathcal H}}|_{x_n =0} = (R'_1 g_n, \cdot , R'_{n-1} g_n, g_n).
\end{array}
\right.
\end{align}
%where $R' = (R'_1, R'_2, \cdots, R'_{n-1})$ is $n-1$ dimensional Riesz transform.
We set $G = g  - w^{{\mathcal H}}|_{x_n =0} = (-R'_1 g_n, \cdots , -R'_{n-1} g_n, 0)$. %so that $G_n =0$.
Let $(w^{{\mathcal S}}, p^{{\mathcal S}})$  be a solution of the following equations in $\R_+ \times (0, \infty)$
\begin{align}\label{boundaryG}
\left\{
\begin{array}{l}
w^{{\mathcal S}}_{t} -\De w^{{\mathcal S}} +\na p^{{\mathcal S}} =0,\qquad
{\rm div  }\, w^{{\mathcal S}} =0,\\
\vspace{-2mm}\\
w^{{\mathcal S}}|_{t =0} =0, \qquad w^{{\mathcal S}}|_{x_n =0} = G.
\end{array}
\right.
\end{align}
Then, $ (w^{{\mathcal H}} + w^{{\mathcal S}}, p^{{\mathcal H}} + p^{{\mathcal S}})$ is  solution of \eqref{Stokes-bvp-200} with external force $f =0$.

Since for $x' \in B^{'}_1$ and $y'  \in B'_2 \setminus B'_\frac32 $, $|x' - y'| \geq \frac12$,
 we have
 $$
 \abs{D_x^k\int_{\Rn} \frac1{|x-y'|^{n-2}} g^{\mathcal S}_n(y') dy'} \leq c \|g^{\mathcal S}_n\|_{L^\infty (B'_2 \setminus B'_\frac32)}
 $$
 for $x^{'} \in B^{'}_1$ and $k \geq 0$.
We note that for $1 \leq   \te \leq  \infty$ and $0 < r < 1$,
\begin{align}\label{0310-1}
\notag\|D_x^2 w^{{\mathcal H}}\|_{L^\te( Q^+_r ) } & = \| g_n^{\mathcal T}\|_{L^\te (1 -r^2, 1)} \| D_x^3\int_{\Rn} \frac1{|x-y'|^{n-2}} g^{\mathcal S}_n(y') dy'\|_{L^\te( B^+_r)},\\
\notag \|D_t w^{{\mathcal H}}\|_{L^\te( Q^+_r ) } & = \| D_t g_n^{\mathcal T}\|_{L^\te (1-r^2,1)} \| D_x^1\int_{\Rn} \frac1{|x-y'|^{n-2}} g^{\mathcal S}_n(y') dy'\|_{L^\te( B^+_r)},\\
  % \leq  \| D_t g_n^2\|_{L^p (0, 1)}\| g_n^1\|_{L^\infty(A)},\\
 \|D_x p^{{\mathcal H}}\|_{L^\te( Q^+_r ) } & = \| D_t g_n^{\mathcal T}\|_{L^\te (1-r^2,1)} \| D_x^1\int_{\Rn} \frac1{|x-y'|^{n-2}} g^{\mathcal S}_n(y') dy'\|_{L^\te( B^+_r)}.
 %  \leq  \| D_t g_n^2\|_{L^p (0, 1)}\| g_n^1\|_{L^\infty(A)}.
 \end{align}
Since $\| g_n^{\mathcal T}\|_{\dot B^{1 -\frac1{3q}}_{\frac{3q}2 \frac{3q}2}({\mathbb R})} \leq \| g_n^{\mathcal T}\|_{W^1_q({\mathbb R})} < \infty$, we have
\begin{align*}
\| G \|_{\dot B^{2 -\frac{2}{3q}, 1 -\frac1{3q}}_{\frac{3q}2 \frac{3q}2} (\Rn \times {\mathbb R})}
&\leq  c \| g_n \|_{\dot B^{2 -\frac{2}{3q}, 1 -\frac1{3q}}_{\frac{3q}2 \frac{3q}2} (\Rn \times {\mathbb R})} \\
&\leq \| g^{\mathcal S}_n \|_{L^{\frac{3q}2}(\Rn)} \|g_n^{\mathcal T}\|_{\dot B^{1 -\frac1{3q}}_{\frac{3q}2 \frac{3q}2} ({\mathbb R})} + \| g_n^{\mathcal S} \|_{\dot B^{2 -\frac2{3q}}_{\frac{3q}2 \frac{3q}2}(\Rn)}\| g_n^{\mathcal T}\|_{L^{\frac{3q}2} ({\mathbb R})} \\
& < \infty.
\end{align*}
From \cite{LS}, we obtain
\begin{align}\label{0310-2}
 \| D_t w^{{\mathcal S}} \|_{L^{\frac{3q}2}((Q^+_1)}+\| D^2_x w^{{\mathcal S}} \|_{L^{{\frac{3q}2} }((Q^+_1)}+\| D_x p^{{\mathcal S}} \|_{L^{\frac{3q}2}((Q^+_1))}   \leq  c \| G \|_{\dot B^{2 -\frac{2}{3q}, 1 -\frac1{3q}}_{\frac{3q}2 \frac{3q}2} (\Rn \times {\mathbb R})}
%&  \leq  c \| g \|_{ W^{2, 1 }_{pp} (\Rn \times (0, 1)}\\
  < \infty.
\end{align}
Combining  \eqref{0310-1}  and \eqref{0310-2}, we obtain
\begin{align*}
\| \na p\|_{L^q (Q^+_1)}+\| D_t w \|_{L^{q}(Q^+_1)}+\| D^2_x w \|_{L^{\frac{3q}2}(Q^+_1)}  < \infty,\\
\| \na p\|_{L^{r_1} (Q^+_\frac12)} =\infty, \quad  \| D_t w\|_{L^{r_1} (Q^+_\frac12)} = \infty \quad r_1 > q.
\end{align*}
For the last result in \eqref{0711-2}, we use the  decomposition of $w_1 = w_1^{ {\mathcal L}}+ w_1^{ {\mathcal B}} + w_1^{ {\mathcal N}}$ defined in Proposition \ref{lemma0406}.
We note that   for $(x, t) \in B(0, 1) \times (0,
1) $ it is direct that
\begin{align}\label{0927-1}
|\na^2 w_1^{\calN}(x,t)| \leq c\| g_n\|_{L^\infty}.
\end{align}
From \eqref{0927-1} and \eqref{0706-8}, for $1 < r < \infty$, we have
\begin{align}\label{0706-9-2}
\| \na^2( w_1 - w_1^{ {\mathcal B}} )\| _{L^r (Q_1^+)} \leq \| g\|_{L^\infty}.
\end{align}
We note that
\begin{align*}
   D_{x_n} D_{x_n}w^{\calB}_1(x,t) &=   c_n  \int_{0}^t \int_{\Rn} D_s  g_n(y', s) \frac{x_n}{(t -s)^{\frac{n+2}2}}e^{-\frac{x_n^2}{t-s}}
\int_{\Rn} e^{-\frac{|x'-y'-z'|^2}{t-s}}    \frac{z_1}{|z'|^n}   dz' dy' ds\\
& \quad  -  c_n  \int_{0}^t \int_{\Rn} \De'  g_n(y', s) \frac{x_n}{(t -s)^{\frac{n+2}2}}e^{-\frac{x_n^2}{t-s}}
\int_{\Rn} e^{-\frac{|x'-y'-z'|^2}{t-s}}    \frac{z_1}{|z'|^n}   dz'\\
& : = I_1 + I_2.
\end{align*}
Firstly, it follows from \eqref{w12-1} that, for $ |x'| \leq 1$,
\begin{align*}
| I_2(x,t)| \leq  c \| \na_{x'}^2  g_n \|_{L^\infty} < \infty.
\end{align*}
Using Lemma \ref{lemma0709-1}, we divide $I_1 = I_{11} + I_{12}$, where
\begin{align*}
I_{11}(x,t) &  = c_n\int_{0}^t D_{x_n} \Ga_1 (x_n, t-s)D_s  g_n^{\mathcal T}( s)  ds \psi(x'),\\
I_{12}(x,t) & = c_n\int_{0}^t D_{x_n} \Ga_1 (x_n, t-s)D_s  g_n^{\mathcal T}( s) \int_{\Rn} g_n^{\mathcal S}(y') J(x' -y', t-s)dy'  ds,
\end{align*}
where $\psi(x') = \int_{\Rn} \frac{x_1 -y_1}{|x' -y'|^n} g_n^{{\mathcal S}}(y') dy'   $.
We consider first $I_{12}$.
Noting that
\begin{align*}
|I_{12} (x,t)| \leq c \int_0^t \frac{x_n}{t-s} e^{-\frac{x_n^2}{t-s}} |D_s  g_n^{\mathcal T}( s) | ds \|g_n^{\mathcal S}\|_{L^\infty}.
\end{align*}
%Hence, we have
%\begin{align*}
%\|  D_{x_k} D_{x_l}w^{\calB}_1 \|_{L^\infty} \leq \| D_{y'}^2 g_n\|_{L^\infty},\\
%\|  D_{x_k} D_{x_n}w^{\calB,2}_1 \|_{L^\infty} \leq \| D_{y'} g_n\|_{L^\infty}
%\end{align*}
%
%
%Note that
%\begin{align}\label{0620-2}
%\notag  D_{x_k} D_{x_n} w^{\calB,1}_1 (x,t) &  =  c \int_0^t D_{x_n} D_{x_n} \Ga_1(x_n, t-s)   g_n^{\mathcal T}(s) ds D_{x_k}\psi(x')\\
%& =  c \Ga_1 * D_t g_n^{{\mathcal T}} (x_n, t)   \psi_i(x').
%\end{align}
%From \eqref{0620-2},  for $1 \leq k, \, l \leq n-1$, we have
%\begin{align*}
%\|D_{x_k} D_{x_l} w^{313}_i\|_{L^\infty (B_1 \times (0, 1))} \leq c \| g_n^{{\mathcal T}} \|_{L^\infty(B_1 \times (0, 1))}.
%\end{align*}
Using the integral Minkowski's inequality, we have
\begin{align*}
\bke{\int_{B_1} |I_{12} (x,t)|^{p_0} dx}^\frac1{p_0}
& \leq c \int_0^t \frac1{t -s } |  D_sg^{{\mathcal T}}_n(s)| \bke{ \int_0^1 x_n^{p_0} e^{-\frac{x_n^2}{t-s}} dx_n}^{\frac1{p_0}} ds \|g_n^{\mathcal S}\|_{L^\infty}\\
& \leq c \int_0^t  (t -s)^{-\frac12 + \frac1{2p_0}} | D_s g^{{\mathcal T}}_n(s)| ds \|g_n^{\mathcal S}\|_{L^\infty}.
\end{align*}
%Hence,  for $\frac1{p_0} +1 = \frac1{p} +\frac12 -\frac1{2p_0} ( %p_0 = \frac{3p}{2 -p})$,
Due to Sobolev-Littlewood-Hardy inequality and Holder inequality, we obtain
\begin{align*}
\| I_{12}\|_{L^{p_0} (Q_1^+)} & \leq  c\| D_t g_n^{{\mathcal T}}  \|_{L^{q}} \|g_n^{\mathcal S}\|_{L^\infty}, \quad 1 < q < 2,\quad p_0 = \frac{3q}{2 -q},\\
\| I_{12} \|_{L^\infty_t L^{p_0}_x (Q_1^+ ))} & \leq  c\| D_t g_n^{{\mathcal T}}  \|_{L^{q}} \|g_n^{\mathcal S}\|_{L^\infty}, \quad 2 \leq q < \infty, \quad p_0 > q.
\end{align*}
On the other hand, since $\int_0^tD_{x_n} \Ga_1(x_n, t-s) g_n^{{\mathcal T}}(s) ds|_{x_n =0} = g_n^{{\mathcal T}}(t)$,
from the result of boundary value problem of heat equation and trace theorem in half space ((3) of Proposition \ref{prop2}), we have that for $ 1 < r < \infty$,
\begin{align}\label{0710-4}
\notag 2\| D_t D_{x_n}  \Ga_1 * g_n^{{\mathcal T}} \|_{ L^{r}({\mathbb R} \times (0, \infty))} &  = \| D_t D_{x_n} \Ga_1 * g_n^{{\mathcal T}} \|_{ L^{r} ({\mathbb R} \times (0, \infty))} +
\| D^3_{x_n}  \Ga_1 * g_n^{{\mathcal T}}\|_{ L^{r}({\mathbb R} \times (0, \infty))} \\
 & \geq c  \| g_n^{{\mathcal T}} \|_{\dot B^{ 1 -\frac1{2r}}_{r r} ({\mathbb R})}.
\end{align}
Since $\| g_n^{{\mathcal T}} \|_{\dot B^{ 1 -\frac1{2r_2}}_{r_2 r_2} ({\mathbb R})}  = \infty$ for $\frac{3q}2 < r_2$,
it follows from \eqref{0710-3}, \eqref{0710-2},\eqref{0710-1}, \eqref{0730-3-1} and \eqref{0710-4} that
\begin{align}
\|I_{11}\|_{L^{r_2} (Q^+_1)} & \geq c  \| g_n^{{\mathcal T}} \|_{\dot B^{ 1 -\frac1{2r_2}}_{r_2 r_2} ({\mathbb R})} \|\psi\|_{L^{r_2} ( B'_\frac12)}-c = \infty.
\end{align}
Thus, we complete the proof of Theorem \ref{maintheorempressure-SS}.
\qed

Following similar computations in the proof of Theorem \ref{maintheorempressure-SS}, we obtain the following corollary. Although its verification is not difficult, for clarity, we present its details.
%{\color{red}$\bullet$}\footnote{ the proof of the corollary needs to be modified.}
\begin{coro}\label{lemm1222-1}
Let $1 < q< \infty$ and   $ g_n^{\mathcal S}$ and $g_n^{\mathcal T}$ satisfy the assumption \eqref{1222-2}. Let $p_0 = \frac{3q}{2 -q}$ if $1 < q < 2$ and $p_0$ be a any number larger than $q$ if $q \geq 2$. Then, the constructed solution $w$ of the Stokes equations \eqref{Stokes-10}  satisfies
\begin{align}\label{coro0927-1}
\| D_x w \|_{L^{p_0}(Q^+_1)}  < \infty.
\end{align}
In case that $1 < q <2$, if $g_n^{\mathcal T}$  is taken as the function in Remark \ref{example-0924} and if $r> p_0$, then
\begin{align}\label{coro0923-1}
\| D_x w \|_{L^{r}((Q^+_1)}  = \infty.
\end{align}
\end{coro}

\begin{proof}
We  recall the  decomposition of $w = w^{ {\mathcal L}}+ w^{ {\mathcal B}} + w^{ {\mathcal N}}$ defined in Proposition \ref{lemma0406}. Flowing similarly the estimate  \eqref{0706-9-2},    for $ 1 < r < \infty  $, we have
\begin{align}\label{0923-3}
\| \na( w - w^{ {\mathcal B}} )\| _{L^r (Q_1^+)} \leq \| g\|_{L^\infty}.
\end{align}
Here we compute only the normal derivative, since we can show that tangential derivatives are rather easily controlled.
%Note that for $1 \leq i \leq n-1$, we have
%\begin{align*}
%D_{x'}w^{\calB}_i(x,t) &=   c_n  \int_{0}^t \int_{\Rn} D_{y'}  g_n(y', s) \frac{x_n}{(t -s)^{\frac{n+2}2}}e^{-\frac{x_n^2}{t-s}}
%\int_{\Rn} e^{-\frac{|x'-y'-z'|^2}{t-s}}    \frac{z_i}{|z'|^n}   dz' dy' ds.
%\end{align*}
%Using the integral Minkowski's inequality, we have
%\begin{align*}
%(\int_{B_1} |D_{x'}w^{\calB}_i(x,t)|^{r} dx)^\frac1{r}
%& \leq c \int_0^t \frac1{t -s } |  g^{{\mathcal T}}_n(s)| ( \int_0^1 x_n^{r} e^{-\frac{x_n^2}{t-s}} dx_n)^{\frac1{r}} ds \|D_{y'}g_n^{\mathcal S}\|_{L^\infty}\\
%& \leq c \int_0^t  (t -s)^{-\frac12 + \frac1{2r}} |  g^{{\mathcal T}}_n(s)| ds \|D_{y'}g_n^{\mathcal S}\|_{L^\infty}\\
%& \leq c   t^{\frac12 + \frac1{2r}} \|  g^{{\mathcal T}}_n \|_{L^\infty(0, \frac12)}  \|D_{y'} g_n^{\mathcal S}\|_{L^\infty}.
%\end{align*}
%Hence, we have
%\begin{align}\label{0923-4}
%\|D_{x'}w^{\calB}_i\|_{L^\infty(0, \frac12; L^r (B^+_\frac12))}
% \leq c    \|  g^{{\mathcal T}}_n \|_{L^\infty(0, \frac12)}  \|D_{y'} g_n^{\mathcal S}\|_{L^\infty}.
%\end{align}
Thus, taking the normal derivative to $w^{\calB}_i$, we compute
\begin{align*}
 D_{x_n}w^{\calB}_i(x,t) &=   c_n  \int_{0}^t \int_{\Rn} D_s  g_n(y', s) \frac{1}{(t -s)^{\frac{n}2}}e^{-\frac{x_n^2}{t-s}}
\int_{\Rn} e^{-\frac{|x'-y'-z'|^2}{t-s}}    \frac{z_i}{|z'|^n}   dz' dy' ds\\
& \quad  +  c_n  \int_{0}^t \int_{\Rn} \De'  g_n(y', s) \frac{1}{(t -s)^{\frac{n}2}}e^{-\frac{x_n^2}{t-s}}
\int_{\Rn} e^{-\frac{|x'-y'-z'|^2}{t-s}}    \frac{z_i}{|z'|^n}   dz'\\
& : = I_1 + I_2.
\end{align*}
Due to \eqref{w12-1},  we have
\begin{align*}
\| I_2 \|_{L^\infty(Q^+_1)} \leq  c \| \na_{x'}^2  g_n \|_{L^\infty} < \infty.
\end{align*}
Using Lemma \ref{lemma0709-1}, we divide $I_1 = I_{11} + I_{12}$, where
\begin{align*}
I_{11}(x,t) &  = c_n\int_{0}^t   \Ga_1 (x_n, t-s)D_s  g_n^{\mathcal T}( s)  ds \psi(x'),\\
I_{12}(x,t) & = c_n\int_{0}^t   \Ga_1 (x_n, t-s)D_s  g_n^{\mathcal T}( s) \int_{\Rn} g_n^{\mathcal S}(y') J(x' -y', t-s)dy'  ds,
\end{align*}
where $\psi(x') = \int_{\Rn} \frac{x_1 -y_1}{|x' -y'|^n} g_n^{{\mathcal S}}(y') dy'   $.
It follows from Lemma \ref{lemma0709-1} that for $1 < r < \infty$,
\begin{align*}
\| I_{12}(t)\|_{L^r(B^+_1)} &  \leq c \int_0^t |D_s  g_n^{\mathcal T}( s) | \big(\int_0^1 e^{-r\frac{x_n^2}{t-s}} dx_n \big)^\frac1r \|g_n^{\mathcal S}\|_{L^\infty}\\
&  \leq c \int_0^t (t -s)^{\frac1{2r}} |D_s  g_n^{\mathcal T}( s) | \|g_n^{\mathcal S}\|_{L^\infty}\\
& \leq c    \| D_t g^{{\mathcal T}}_n \|_{L^p(0, 1)}  \|D_{y'} g_n^{\mathcal S}\|_{L^\infty}.
\end{align*}
%Hence, we have
%\begin{align*}
%\|I_{12}\|_{ L^r (Q^+_\frac12)}
% \leq c    \| D_t g^{{\mathcal T}}_n \|_{L^p(0, \frac12)}  \|D_{y'} g_n^{\mathcal S}\|_{L^\infty}.
%\end{align*}
Continuing computations for $I_{11}$,
\begin{align*}
\| I_{11}(t)\|_{L^{p_0}(B^+_\frac12)} &  \leq c \int_0^t (t-s)^{-\frac12} |D_s  g_n^{\mathcal T}( s) | \big(\int_0^\frac12 e^{-p_0\frac{x_n^2}{t-s}} dx_n \big)^{\frac1{p_0}} \|g_n^{\mathcal S}\|_{L^\infty}\\
&  \leq c \int_0^t (t -s)^{-\frac12 +\frac1{2p_0}} |D_s  g_n^{\mathcal T}( s) | \|g_n^{\mathcal S}\|_{L^\infty}.
\end{align*}
Due to Young's inequality, we obtain
\begin{align*}
\|I_{11}\|_{ L^{p_0} (Q^+_1)}
 \leq c    \| D_t g^{{\mathcal T}}_n \|_{L^p(0, 1)}  \|D_{y'} g_n^{\mathcal S}\|_{L^\infty}.
\end{align*}
Hence, summing up the estimates, for $1 \leq i \leq n-1$, we have
\begin{align}\label{0923-2}
\| \na w_i\|_{L^{p_0} (Q^+_1)} < \infty.
\end{align}
From the first equality of \eqref{1006-3},  we obtain \eqref{0923-2} for $i =n$. Hence, we complete the proof of \eqref{coro0927-1}.
%{\color{red}$\bullet$}\footnote{what is ca?}
It remains to show \eqref{coro0923-1}.
It follows from the above estimate that
\begin{align}\label{0927-2}
\|\na w - I_{11}\| _{L^r (Q_1^+)} \leq \infty \qquad 1 < r < \infty.
\end{align}
Noting that for $ \frac34 < t < 1$,
\begin{align*}
|D_t g^{\mathcal T}_n(t)| & =(t-\frac34)^{-\frac{1}{q}}| \ln^{-1} (t-\frac34)|\bke{(1-\frac{1}{q})-\ln^{-1} (t-\frac34)}\\
& \geq (1-\frac{1}{q})(t-\frac34)^{-\frac{1}{q}}| \ln^{-1} (t-\frac34)|.
\end{align*}
%and thus for any $\epsilon\in (0, \frac{1}{2}(1-\frac{1}{q}))$ if $\frac1{16} < t <\frac1{16}+e^{-\frac{1}{\epsilon}} $, then we see that
%\[
%\abs{D_t g^{\mathcal T}_n(t)}\ge  (1-\frac{1}{q})(t-\frac{1}{16})^{-\frac{1}{q}}\abs{\ln^{-1} (t-\frac{1}{16})}-(t-\frac{1}{16})^{-\frac{1}{q}}\abs{\ln^{-2} (t-\frac{1}{16})}
%\]
%\[
%\ge  (1-\frac{1}{q}-\epsilon)(t-\frac{1}{16})^{-\frac{1}{q}}\abs{\ln^{-1} (t-\frac{1}{16})}
%\ge   \frac{1}{2}(1-\frac{1}{q})(t-\frac{1}{16})^{-\frac{1}{q}}\abs{\ln^{-1} (t-\frac{1}{16})}
%\]
Hence, for  $ \frac34< t <1 $, we have
\begin{align*}
\abs{\int_0^t   \Ga_1 (x_n, t-s)D_s  g_n^{\mathcal T}( s)  ds}
&\geq c
\int_{\frac34}^{\frac12(t + \frac34)}   (t -s)^{-\frac12} e^{-\frac{x_n^2}{4(t-s)}} (s-\frac34)^{-\frac1q} \abs{\ln^{-1} (s-\frac34)} ds\\
&\geq c  (t-\frac34)^{-\frac12} e^{-\frac{x_n^2}{2(t-\frac34)}}  \int_{ \frac34}^{\frac12(t + \frac34)}  (s-\frac34)^{-\frac1q} \abs{\ln^{-1} (s-\frac34)} ds\\
& = c  (t-\frac34)^{-\frac12} e^{-\frac{x_n^2}{2(t-\frac34)}}  \int^{\frac12 (t- \frac34)}_0  s^{-\frac1q} \abs{\ln^{-1} (s )} ds.
\end{align*}
By L'hopital's Theorem, we have $\displaystyle\lim_{a \ri 0} \frac{\int^a_0  s^{-\frac1q} |\ln^{-1} (s )| ds}{a^{1-\frac1q} |\ln^{-1} (a )| } = \frac{q-1}{q} $, we have
\begin{align*}
\abs{\int_0^t   \Ga_1 (x_n, t-s)D_s  g_n^{\mathcal T}( s)  ds}
&\geq c (t-\frac34)^{\frac12 -\frac1q} e^{-\frac{x_n^2}{2(t-\frac34)}}  \abs{\ln^{-1} (t-\frac34)}.
\end{align*}
Hence, for $r > p_0$,  we have
\begin{align*}
\|I_{11}\|^r_{ L^{r} (Q^+_\frac12)} \geq c \int_{\frac34}^1(t-\frac34)^{(\frac12 -\frac1q)r +\frac12} \abs{\ln^{-r} (t-\frac34)} dt \| \psi\|^r_{L^r(B'_\frac12)} =\infty.
\end{align*}
Thus,  combining the estimate \eqref{0927-2}, we have \eqref{coro0923-1}.
We complete the proof of Corollary  \ref{lemm1222-1}.
\end{proof}

\section{Proof of main results %Theorem \ref{maintheorem-SS} and Theorem \ref{maintheorempressure-SS}
for Navier-Stokes equations}
\label{proof-ns}
\setcounter{equation}{0}

Previously, we presented the proofs of Theorem \ref{maintheorem-SS} and Theorem \ref{maintheorempressure-SS} for the Stokes equations.
In this section we complete the proofs of those results by providing the details for the case of Navier-Stokes equations.
We start with the case of Theorem \ref{maintheorem-SS}

\subsection{Proof of Theorem \ref{maintheorem-SS} for Navier-Stokes equations}
\label{ns-thm11}
%\footnote{The proof is very similar to that of JDE. Can we refer to that one in JDE?}

Let $g$ be a boundary data
defined in \eqref{0502-6} and \eqref{boundarydata}, and $w$ be a solution of the Stokes
equations \eqref{Stokes-bvp-200}-\eqref{Stokes-bvp-210} with $f =0$  defined by
\eqref{rep-bvp-stokes-w}. Let $\phi_1 \in C_c^\infty({\mathbb R}^{n})$ be a  cut-off function satisfying $ \phi_1 \geq 0$,  ${\rm supp} \, \phi_1 \subset B_2$ and $\phi_1 \equiv 1 $ in $B_1$. Also, let $\phi_2 \in C_c^\infty(-\infty, \infty)$ be a  cut-off function satisfying $ \phi_2 \geq 0$,  ${\rm supp} \, \phi_2 \subset (-2, 2)$ and $\phi_2 \equiv 1 $ in $(-1, 1)$. We set $\phi (x,t) =\phi_1 (x) \phi_2 (t)$.

We define $W = \phi w$. Then, it is direct that $ W = w $ in $ Q^+_1 $.
Furthermore, since supports of $g_n$ and $\phi$ are disjoint, we note that $W|_{x_n =0 } =0$ and $W|_{t =0} =0$.
By the result of  Proposition \ref{lemma0406}, we observe that
\begin{align}
\label{uc1-1}  \| W\|_{L^r (\R_+ \times (0, 1))} \leq c\| w\|_{L^r (B_2^+ \times (0, 4))} \leq c\| g\|_{L^\infty (\Rn \times (0, \infty))} \leq ca
% \| \na W\|_{L^{p_0}(B_r \times (t_0, t_0 + r^2))} = \infty
\end{align}
for all $ 1 \leq r \leq \infty$, where $a >0$ is defined in \eqref{0502-6}.

We consider the following perturbed Navier-Stokes equations in
$\R_+ \times (0, 1)$:
\begin{equation}\label{CCK-Feb7-10}
v_t-\Delta v+\nabla q+{\rm div}\,\left(v\otimes v+v\otimes
W+W\otimes v\right)=-{\rm div}\,(W\otimes W), \quad {\rm
div} \, v =0
\end{equation}
with homogeneous initial and boundary data, i.e. $v(x,0)=0$ and $ v(x,t)=0$ on $x_n=0.$
%\begin{equation}\label{pnse-bdata-20}
%v(x,0)=0,\qquad v(x,t)=0 \,\,\mbox{on} \,\,\{x_n=0\}.
%\end{equation}
Our aim is to establish the existence of solution  $v$ for
\eqref{CCK-Feb7-10} satisfying $v \in L^{r} (\R_+ \times (0,1)) \cap L^{\infty} (\R_+ \times (0,1))$
and $\na v \in L^r (\R_+ \times (0,1) )$ for all $ n+2 < r< \infty$. In order to do that, we
consider the iterative scheme for \eqref{CCK-Feb7-10}, which is
given as follows: For a positive integer $m\ge 1$
\begin{align*}%\label{CCK-Feb7-20}
&v^{m+1}_t-\Delta v^{m+1}+\nabla q^{m+1}=-{\rm
div}\,\left(v^{m}\otimes v^{m}+v^{m}\otimes W+W\otimes
v^{m}+W\otimes W \right),\\
& \qquad \qquad \qquad \qquad \qquad {\rm div} \, v^{m+1} =0
\end{align*}
with homogeneous initial and boundary data, i.e. $v^{m+1}(x,0)=0$ and
$v^{m+1}(x,t)=0$ on $\{x_n=0\}$.
We set $v^1=0$.  By Proposition \ref{proposition1},  we have
\begin{align*}
\| v^2(t)\|_{L^r (\R_+)} & \leq c \int_0^t ( t -s)^{-\frac12} \| W (s) \otimes W(s) \|_{L^r (\R_+)} ds.
%& \leq c  \|W \|_{L^\infty (\R_+ \times (0, 1))}\int_0^t ( t -s)^{-\frac12} \| W (s)\|_{L^p (\R_+)}  ds.
\end{align*}
From now on, we denote $Q_+:=\R_+ \times (0,1)$, unless any confusion is to be expected.
By Young's inequality, we have
\begin{align}\label{1202-1}
\| v^2\|_{L^r (Q_+)} \leq  c \|W \otimes W \|_{L^r (Q_+)} \leq  c \|W   \|_{L^r (Q_+)}   \|W   \|_{L^\infty (Q_+)}.
\end{align}
Since $r > n+2$, according to Proposition \ref{theoexternel-boundary}, we have
\begin{align}\label{est-v2-10}
  \|  v^2\|_{L^{\infty} (Q_+)}  &\leq c     \|  W\otimes W\|_{L^{r} (Q_+)} \leq c     \|  W\|_{L^{r} (Q_+)} \|  W\|_{L^{\infty} (Q_+)}.
\end{align}
According to      Proposition \ref{theo0503}, we have
%\begin{align}\label{est-v2-10}
%  \|\na^2  v^2\|_{L^p (\R_+ \times (0, 1))}  + \|D_t  v^2\|_{L^{p} (\R_+ \times (0, 1))}   &\leq c     \| \na W\cdot W\|_{L^{p} (\R_+ \times (0, 1))}\\
%&\leq c     \| \na W\|_{L^{p} (\R_+ \times (0, 1))} \|  W\|_{L^{\infty} (\R_+ \times (0, 1))}.
%\end{align}
\begin{align}\label{est-v2-10}
  \|\na  v^2\|_{L^{r} (Q_+)}  &\leq c     \|  W\otimes W\|_{L^{r} (Q_+)} \leq c     \|  W\|_{L^{r} (Q_+)} \|  W\|_{L^{\infty} (Q_+)}.
\end{align}
By \eqref{uc1-1}, we have  $ A :=    \|W\|_{L^r (Q_+) } + \|W\|_{L^\infty (Q_+) }  \leq 2ca < 1$, where $a>0$ is defined in \eqref{0502-6}. Taking $a >0$ small such that
$A<\frac{1}{4c}$, where $c$ is the constant in
\eqref{1202-1}-\eqref{est-v2-10} such that
\begin{align*}
\| v^2\|_{L^r (Q_+)}+ \| v^2\|_{L^\infty (Q_+)} + \|\na  v^2\|_{L^r( Q_+)}   < \frac{3c}2A^2 < A.
\end{align*}
Suppose that for $m \geq 2$,
\begin{align}\label{0823-1}
\|\na v^m\|_{L^r (Q_+) } + \|v^m\|_{L^r (Q_+) } + \|v^m\|_{L^\infty (Q_+) } < A.
\end{align}
Then, iterative arguments show that
\begin{align}\label{0531-1}
\notag&\| \na v^{m+1}\|_{L^r (Q_+) }+\|  v^{m+1}\|_{L^r (Q_+) }+\|  v^{m+1}\|_{L^\infty (Q_+) } \\
 \notag & \quad \leq c \big(
\||v^m|^2+|v^mW|+|W|^2\|_{L^r (Q_+) }  \big) \\
\notag &\quad \leq  c\left(\|v^m\|_{L^r (Q_+) } + \|W\|_{L^r (Q_+) }\right) \Big(  \|v^m\|_{L^\infty (Q_+) } + \|W\|_{L^\infty (Q_+) }   \Big)\\
 & \quad   \leq 4c A^2 <A.
\end{align}
Hence, by mathematical induction, \eqref{0823-1} holds for all $m \geq 2$.
%Similarly, we note that
%\begin{align}\label{0531}
%\notag&\|  v^{m+1}\|_{L^p (\R_+ \times (0,1)) }\\
%\notag & \quad \leq c \big(
%\||v^m|^2+|v^mW|+|W|^2\|_{L^p (\R_+ \times (0,1)) }  \big) \\
%\notag &\quad \leq  2c\left(\|v^m\|_{L^p (\R_+ \times (0,1)) } + \|W\|_{L^p (\R_+ \times (0,1)) }\right) \Big(  \|v^m\|_{L^\infty (\R_+ \times (0,1)) } + \|W\|_{L^\infty (\R_+ \times (0,1)) }   \Big)\\
% & \quad   \leq 4c A^2 <A.
%\end{align}
%
%\begin{align}\label{0531}
%\notag&\|  v^{m+1}\|_{L^\infty (\R_+ \times (0,1)) }\\
%\notag & \quad \leq c \big(
%\||v^m|^2+|v^mW|+|W|^2\|_{L^p (\R_+ \times (0,1)) }  \big) \\
%\notag &\quad \leq  2c\left(\|v^m\|_{L^p (\R_+ \times (0,1)) } + \|W\|_{L^p (\R_+ \times (0,1)) }\right) \Big(  \|v^m\|_{L^\infty (\R_+ \times (0,1)) } + \|W\|_{L^\infty (\R_+ \times (0,1)) }   \Big)\\
% & \quad   \leq 4c A^2 <A.
%\end{align}

We denote, for convenience, $V^{m+1}:=v^{m+1}-v^{m}$ and $Q^{m+1}:=q^{m+1}-q^{m}$ for $m\ge
1$. We then see that $(V^{m+1}, Q^{m+1})$ solves
\[
V^{m+1}_t-\Delta V^{m+1}+\nabla Q^{m+1}=-{\rm
div}\,\left(V^{m}\otimes v^{m}+v^{m-1}\otimes V^{m}+V^{m}\otimes
W+W\otimes V^{m}\right),
\]
\[
{\rm div} \, V^{m+1} =0,
\]
with homogeneous initial and boundary data, i.e. $V^{m+1}(x,0)=0$
and $V^{m+1}(x,t)=0$ on $\{x_n=0\}$. Taking sufficiently small
$a>0$ such that $A < \frac1{6c}$, we obtain from \eqref{0531-1} that
\begin{align}\label{est-v2-10-infty}
\notag &\|  V^{m+1}\|_{L^{r} (Q_+ )}+\|  V^{m+1}\|_{L^{\infty} (Q_+)}+\|\na  V^{m+1}\|_{L^{r} (Q_+)}\\
\notag   &     \leq c \big(
V^{m}\otimes v^{m}+v^{m-1}\otimes V^{m}+V^{m}\otimes
W+W\otimes V^{m}\|_{L^r (Q_+)}   \big) \\
\notag & \leq c  \|  V^m\|_{L^{r} (Q_+)}  \big( \|  v^m\|_{L^{\infty} (Q_+)} + \|  v^{m-1}\|_{L^{\infty} (Q_+)} + + \|  W\|_{L^{\infty} (Q_+)}  \big)\\
& \leq \frac12   \|  V^m\|_{L^{r} (Q_+)}.
\end{align}
%\begin{align}\label{est-v2-10}
%  \|\na  V^{m+1}\|_{L^{p} (\R_+ \times (0, 1))}  &\     \leq c \big(
%\| V^m \otimes v^m+| v^{m-1} \otimes V^m|\|_{L^{p}( \R_+ \times (0, \infty))}   \big) \\
%& \leq \frac12   \|  V^m\|_{L^{p} (\R_+ \times (0, 1))}.
%\end{align}
%
This implies that  $(v^m, \na v^m)$ converges to $(v, \na v)$ in
\begin{align*}%\label{1130-1}
\bke{L^r(Q_+) \cap L^\infty (Q_+) }  \times L^r (Q_+)
\end{align*}
such that $v$
solves \eqref{CCK-Feb7-10} with an appropriate distribution $q$.
We then set $u:=v+W$ and $p =\pi + q$, which becomes a  weak solution of the
Navier-Stokes equations in $\R_+ \times (0, 1)$, namely
\[
u_t-\Delta u+\nabla p=-{\rm div}\,\left(u\otimes u\right),\qquad
{\rm div} \, u=0 \quad \mbox{in} \quad Q^+_1
\]
with boundary data $u(x,t)=0$ on $\Sigma = ( B_1 \cap \{ x_n =0\}) \times (0, 1)$ such that
\begin{align}\label{0502-7}
 \| u \|_{L^{\infty} (Q^+_1 )} \leq c,\qquad \| \na u\|_{L^p (Q^+_\frac12 )} = \infty.
\end{align}
This completes the proof of Theorem \ref{maintheorem-SS} for the case of the Navier-Stokes equations.
\qed

\subsection{Proof of Theorem \ref{maintheorempressure-SS} for Navier-Stokes equations}

\label{ns-thm15}
%\setcounter{equation}{0}

%\begin{lemm}\label{lemm1222-1}
%Suppose that $g_n^{\calS}$ and $g_n^{\calT}$ satisfy the assumption \eqref{1222-2}.
%\begin{itemize}
%\item[(1)]
%Then, for $ 1 < p <2$,
%\begin{align*}
%\| D_x w \|_{L^{\frac{3p}{2-p}}(Q^+_1)}  < \infty.
%\end{align*}
%
%\item[(2)]
%If $ 2 \leq p$, then for  $1 < r < \infty$,
%\begin{align*}
%\| D_x w \|_{L^\infty(0, 1;L^{r}(B^+_1))}  < \infty.
%\end{align*}
%
%
%\end{itemize}
%
%\end{lemm}
%We proved Lemma \ref{lemm1222-1} in Appendix \ref{appendixa-2}.

%\footnote{KK: the following is previous version of the lemma {\color{magenta}
%Let $p$, $p_0$,  $ g_n^1$ and $g_n^2$ satisfy the assumption \eqref{1222-2}. Then,
%\begin{align*}
%\| D_x w \|_{L^{p_0}((Q^+_1)}  < \infty.
%\end{align*}
%}}

Let $g_n$ be a boundary data
defined in \eqref{1222-2} and $w$ be a solution of the Stokes
equations \eqref{Stokes-bvp-200}-\eqref{Stokes-bvp-210} with $f =0$  defined by
\eqref{rep-bvp-stokes-w}. Let $\phi \in C_c^\infty (\R \times (0, \infty))$ be a cut-off function defined in Section \ref{ns-thm11}. Let $W = \phi w$ and $P = \phi \pi$ such that $ W = w $   and $P =\pi$ in $Q^+_1$.
%{\color{red} $\bullet$} \footnote{Is it in $ B_{1/4}^+ \times (0, 1/4) $, not $Q^+_\frac12$?}
As before, we denote $Q_+:=\R_+ \times (0,1)$ for simplicity.

Let  $a >0$ be the number defined in \eqref{0502-6}.
From Section \ref{proofss} and Corollary \ref{lemm1222-1}, we have
\begin{align} \label{uc1-1-2}
 %\notag
 \| D_t W\|_{L^q (Q_+)} +  \| D^2_x W\|_{L^{\frac{3q}2} (Q_+)}  + \|\na P\|_{L^q(Q_+)}
%\notag& \qquad   \leq c \big(  \| D_t w\|_{L^{p} (Q^+_\frac12) } +  \| D^2_x w\|_{L^{\frac{3p}2} (Q^+_\frac12)} +  \|\na \pi \|_{L^{p}(Q^+_\frac12)} \big)\\
%  < ca,\\
 + \|D_x W\|_{L^{p_0} (Q_+)} < ca,
\end{align}
where $ p_0 := \frac{ 3q}{ 2-q}$ if $q < 2 $ and $p_0$ is any real number satisfying $p_0 > \frac{3q}2$ if  $q \geq 2$. We note that $p_0 > \frac{3q}2$.
%{\color{red} $\bullet$} \footnote{Why we do need to mention  $p_0 > \frac{3q}2$?}
From \eqref{est-L-tensor} and \eqref{0310-1}, applying Section \ref{proofss}, we have
\begin{align}\label{0913-1}
\|W\|_{L^q_t L^r_x (Q_+)} < \|w\|_{L^q_t L^r_x (Q_1^+)} < ca, \quad \forall r < \infty.
\end{align}
We consider the following perturbed Navier-Stokes equations in
$Q_+$:
\begin{equation}\label{CCK-Feb7-10-pressure-2}
v_t-\Delta v+\nabla q+{\rm div}\,\left(v\otimes v+v\otimes
W+W\otimes v\right)=-{\rm div}\,(W\otimes W), \quad {\rm
div} \, v =0
\end{equation}
with homogeneous initial and boundary data, i.e. $
v(x,0)=0 $ and $v(x,t)=0$  on $x_n=0$.

Our aim is to establish the existence of solution  $v$ for
\eqref{CCK-Feb7-10-pressure-2} satisfying
\begin{align*}
v \in L^{\infty}(Q_+), \qquad D_x v, \,
 D_x^2 v, \, D_t v, \,  \, q \in L^{p_0}  ( Q_+).
\end{align*}
Since the proofs are exactly the same, we only prove for the case of $q < 2$.
%\begin{equation}\label{pnse-bdata-20-2}
%v(x,0)=0,\qquad v(x,t)=0 \,\,\mbox{on} \,\,\{x_n=0\}.
%\end{equation}
%\begin{theo}
%\label{thm-stokes}
%Let    $1<p, q<\infty$ and  $0\leq \al<\frac{1}{2}(1-\frac{1}{p}-\frac{2}{q})$. Let $h\in \dot B^{-2\al-\frac{2}{q}}_{pq}(\R_+)$ with $\mbox{\rm div} \, h=0$ and $ g\in L^q_\al(0,T;\dot{B}^{-\frac{1}{p}}_{pp}(\Rn))$.
%Let  ${\mathcal F}\in
%L^{q_1}_{\al_1}(0,T; L^{p^{{\mathcal H}}}(\R))
%$, ${\al_1}=\al+ \frac{1}{2}-\frac{n}{2}(\frac{1}{p^{{\mathcal H}}}-\frac{1}{p})+\frac{1}{q}-\frac{1}{q_1}$ with $q\geq q_1$ and  $\al\leq \al_1<1-\frac{1}{q_1}$.
%
%Then, there is a weak solution $u\in L^q_\al(0,T;L^p(\R_+))$ to the Stokes equations \begin{align}
%&u_t-\Delta u+\nabla \pi = {\rm div}\,{\mathcal F},\\
%&{\rm div} \, u =0,\\
%&u|_{t =0} =0, \quad u|_{x_n =0} =0
%\end{align}
%with   the following estimate
%\begin{align*}
%\| u\|_{L^q_\al(0,T;L^p(\R_+))}\leq c  \|{\mathcal F}\|_{L^{q_1}_{\al_1}(0,T; L^{p^{{\mathcal H}}}(\R_+))}.
%\end{align*}
%Moreover, the solution is unique in the class $L^q_\al(0,T;L^p(\R_+))$ for $0\leq \al<1-\frac{1}{q}$.
%\end{theo}
We consider the iterative scheme for \eqref{CCK-Feb7-10-pressure-2}, which is
given as follows: For a positive integer $m\ge 1$
\begin{align*}%\label{CCK-Feb7-20}
&v^{m+1}_t-\Delta v^{m+1}+\nabla q^{m+1}=-{\rm
div}\,\left(v^{m}\otimes v^{m}+v^{m}\otimes W+W\otimes
v^{m}+W\otimes W \right),\\
& \qquad \qquad \qquad \qquad \qquad {\rm div} \, v^{m+1} =0
\end{align*}
with homogeneous initial and boundary data, i.e. $v^{m+1}(x,0)=0$ and
$v^{m+1}(x,t)=0$ on $\{x_n=0\}$.
We set $v^1=0$.   From the well-known result of initial-boundary value problem for Stokes equations in half space, we have
\begin{align}\label{est-v2-10-2-1}
 \notag \|\na^2  v^2\|_{L^{p_0} (Q_+)}  + \|D_t  v^2\|_{L^{p_0} (Q_+)}   &\leq c     \| \na W\cdot W\|_{L^{p_0} (Q_+)}\\
&\leq c     \| \na W\|_{L^{p_0} (Q_+)} \|  W\|_{L^{\infty} (Q_+)}.
\end{align}
According to  Proposition \ref{theo0503}, Proposition \ref{thm-stokes} and  Proposition \ref{theoexternel-boundary-2},   we have
\begin{align}\label{est-v2-10-2}
\notag  \|\na  v^2\|_{L^{p_0} (Q_+)}  &\leq c     \|  W\otimes W\|_{L^{p_0} (Q_+)}\leq c     \|  W\|^2_{L^{\infty} (Q_+)},\\
  \|  v^2\|_{L^{p_0} (Q_+)}  &\leq c     \|  W\otimes W\|_{L^{p_0} (Q_+)}\leq c     \|  W\|^2_{L^{\infty} (Q_+)}.
\end{align}
Take $r_0< \infty$ satisfying $ \frac2q +\frac{n}{r_0} < 2$.
From Proposition \ref{prop0803-1},  Young's inequality and Holder inequality,  we have
\begin{align*}
  |  v^2(x,t)|  &\leq c   \int_0^t (t -s)^{-\frac{n}{2r_0}} \|  \na W(s) \cdot W(s)\|_{L^{r_0}  (\R_+  )}  ds\\
  &\leq c   t^{-\frac{n}{2r_0}+1 -\frac1q} \|  \na W \cdot W\|_{L^q_tL^{r_0}_x  (Q_+ )}.
\end{align*}
Hence, we obtain
\begin{align}\label{0913-2}
\| v^2\|_{L^\infty(Q_+)} \leq c \|  \na W \|_{L^q_tL^{r_0}_x  (Q_+  )} \|W\|_{L^\infty(Q_+)}.
\end{align}
Moreover, from Proposition \ref{prop0803-1}, Proposition \ref{prop0803-2},  Young's inequality and Hardy-Littlewood-Sobolev's inequality,  we have
\begin{align}\label{0913-3}
\| v^2\|_{L^{2q}_t L^{2r_0}_x (Q_+)} &  \leq c \| W \otimes W\|_{L^q_t L^{r_0}_x (Q_+)}  \leq c \| W \|^2_{L^{\infty} (Q_+)}.
\end{align}
From \eqref{est-L-tensor} and \eqref{0310-1}, applying Section \ref{proofss}, we obtain $\|W\|_{L^{\infty}( \R_+ \times (0,1)) } \leq \|w\|_{L^{\infty}(Q^+_1 )} \leq ca$, where $ a>0$ is defined in \eqref{0502-6}. In Corollary \ref{lemm1222-1}, we showed that $\| \na w\|_{L^{p_0} (Q^+_\frac12 )} < ca$. Then, from \eqref{0913-1},  we have
\begin{align*}
 A : &=   \| \na W\|_{L^{p_0} (Q_+)} + \|  \na W \|_{L^q_tL^{r_0}_x  (Q_+  )} +       \|W\|_{L^{\infty}( Q_+) }
\leq ca.
\end{align*}
 Taking $a >0$ small such that
$ca < A<\min(\frac{1}{4c},\frac12)$, where $c$ is the constant in \eqref{est-v2-10-2-1}, \eqref{est-v2-10-2}, \eqref{0913-2} and \eqref{0913-3}, we have
\begin{align*}
&\|\na^2  v^2\|_{L^{p_0} (Q_+)}  + \|D_t  v^2\|_{L^{p_0} (Q_+)} + \|\na  v^2\|_{L^{p_0} (Q_+)} + \|  v^2\|_{L^{p_0} (Q_+)}\\
& \quad    + \|  v^2\|_{L^{\infty} (Q_+)} + \| \na v^2\|_{L^q_tL^{r_0}_x(Q_+)} + \|  v^2\|_{L^{2q}_tL^{2r_0}_x(Q_+)} < A.
\end{align*}
Suppose that for $m \geq 2$,
\begin{align}\label{0823-1-1}
\notag &\| \na^2 v^{m}\|_{L^{p_0}( Q_+) } + \| D_t v^{m}\|_{L^{p_0}( Q_+) } +  \|\na  v^{m}\|_{L^{p_0}( Q_+)}+  \|  v^{m}\|_{L^{p_0} ( Q_+)}\\
& \quad  +  \|  v^{m}\|_{L^{\infty}( Q_+)}    + \| \na v^m\|_{L^q_tL^{r_0}_x(Q_+)} + \|  v^m\|_{L^{2q}_tL^{2r_0}_x( Q_+)}  < A.
\end{align}
Then, first we obtain
\begin{align}\label{0531-1-2}
\notag&\| \na^2 v^{m+1}\|_{L^{p_0}( Q_+) } + \| D_t v^{m+1}\|_{L^{p_0}( Q_+) }\\
 \notag \leq &c \big(
\|\na v^m \otimes v^m+|\na v^m \otimes W|+ |\na W \otimes v^m|+|\na W\otimes W|\|_{L^{p_0}( Q_+)}   \big) \\
%\notag
\leq&  2c\big(\|\na v^m\|_{L^{p_0}( Q_+)  }+\|\na W\|_{L^{p_0} ( Q_+)  }\big)
\big(\|v^m\|_{L^{\infty}( Q_+)  }+\|W\|_{L^{\infty} ( Q_+)  }\big)
%\\& \quad
 \leq 4c A^2 <\frac12 A.
\end{align}
Similarly, we have
\begin{align*}
&  \|  v^{m+1}\|_{L^{p_0} (Q_+  )}+\|\na  v^{m+1}\|_{L^{p_0} (Q_+)}\\
       \leq &c  \big(\| v^m  \otimes v^m|  + | v^m  \otimes W|+ | W \otimes v^m|+| W\otimes W|\|_{L^{p_0}( Q_+)}   \big) \\
\leq &c   \big( \|  v^m\|_{L^{p_0} ( Q_+)} \|  v^m\|_{L^{\infty} (Q_+)} + \|  v^m\|_{L^{p_0} ( Q_+)} \|  W\|_{L^{\infty} ( Q_+)} + \|  W\|^2_{L^{p_0} ( Q_+)} \big)
%\\
\leq A^2 < \frac12A.
\end{align*}
Continuing computations for $L^{\infty}$ (from \eqref{0913-2}), we have
\begin{align}\label{0913-2-1}
\notag \|  v^{m+1}\|_{L^{\infty} (Q_+ )}
%\\ \notag
 \leq &c   \big( \| \na v^m\|_{L^q_tL^{r_0} ( Q_+)} \|  v^m\|_{L^{\infty} ( Q_+)}+\| \na v^m\|_{L^q_tL^{r_0} ( Q_+)} \|  W\|_{L^{\infty} ( Q_+)}\\
%\notag
& \quad   + \| \na W\|_{L^q_tL^{r_0} ( Q_+)} \|  W\|_{L^{\infty} ( Q_+)} \big)
%\\
\leq A^2 < \frac12A.
\end{align}
For mixed norm, we compute likewise, that is,
\begin{align*}
&\|\na  v^{m+1}\|_{L^{q}_tL^{r_0}_x (Q_+)}\\
    \leq &c  \big(\| v^m  \otimes v^m|  + | v^m  \otimes W|+ | W \otimes v^m|+| W\otimes W|\|_{L^{q}_tL^{r_0}_x (Q_+)}   \big) \\
\leq &c   \big( \|  v^m\|^2_{L^{2q}_tL^{2r_0}_x (Q_+)} + \|  v^m\|_{L^{2q}_tL^{2r_0}_x (Q_+)} \|  W\|_{L^{2q}_tL^{2r_0}_x (Q_+)} + \|  W\|^2_{L^{2q}_tL^{2r_0}_x (Q_+)} \big)
%\\
\leq A^2 < \frac12A.
\end{align*}
In the same way as above, we get
\begin{align*}
&\| v^{m+1}\|_{L^{2q}_tL^{2r_0}_x (Q_+)}\\
    \leq & c  \big(\| |v^m  \otimes v^m|  + | v^m  \otimes W|+ | W \otimes v^m|+| W\otimes W|\|_{L^{q}_tL^{r_0}_x (Q_+)}   \big) \\
 \leq &c   \big( \|  v^m\|^2_{L^{2q}_tL^{2r_0}_x (Q_+)} + \|  v^m\|_{L^{2q}_tL^{2r_0}_x (Q_+)} \|  W\|_{L^{2q}_tL^{2r_0}_x (Q_+)} + \|  W\|^2_{L^{2q}_tL^{2r_0}_x (Q_+)} \big)
 %\\
\leq A^2 < \frac12A.
\end{align*}
Hence, by mathematical induction, \eqref{0823-1-1} holds for all $m \geq 2$.

Next, we denote $V^{m+1}:=v^{m+1}-v^{m}$ and $Q^{m+1}:=q^{m+1}-q^{m}$ for $m\ge
1$. We then see that $(V^{m+1}, Q^{m+1})$ solves
\[
V^{m+1}_t-\Delta V^{m+1}+\nabla Q^{m+1}=-{\rm
div}\,\left(V^{m}\otimes v^{m}+v^{m-1}\otimes V^{m}+V^{m}\otimes
W+W\otimes V^{m}\right),
\]
\[
{\rm div} \, V^{m+1} =0,
\]
with homogeneous initial and boundary data, i.e. $V^{m+1}(x,0)=0$
and $V^{m+1}(x,t)=0$ on $\{x_n=0\}$. Taking sufficiently small
$\al>0$ such that $A < \frac1{6c}$, from \eqref{0531-1-2}, we obtain
\begin{align}\label{0531-1-3}
\notag&\| \na^2 V^{m+1}\|_{L^{p_0}( Q_+) } + \|D_t V^{m+1}\|_{L^{p_0}( Q_+) }\\
 %\notag & \quad \leq c \big(
%\|\na V^m \otimes v^m+\na v^{m-1} \otimes V^m +V^{m}\otimes
%W+W\otimes V^{m}\|_{L^{p_0} (\R_+ \times (0,1))}   \big) \\
\notag \leq & 2c\big(\|\na V^m\|_{L^{p_0}( Q_+)  }+\|V^m\|_{L^{\infty} ( Q_+)  }\big)\\
\notag & \quad \times
\big(\|v^m\|_{L^{\infty}( Q_+)  }+\|\na v^{m-1}\|_{L^{p_0} ( Q_+)  } + \|W\|_{L^{\infty}( Q_+)  } + \|\na W\|_{L^{p_0}( Q_+)  }\big)\\
 \leq &\frac12 \big(\|\na V^m\|_{L^{p_0}( Q_+)  }+\|V^m\|_{L^{\infty} ( Q_+)  }\big).
\end{align}
Similarly, we have
\begin{align}\label{est-v2-10-infty-2}
\notag &\|  V^{m+1}\|_{L^{p_0} (Q_+)}+\|\na  V^{m+1}\|_{L^{p_0} (Q_+)}\\
 \notag      \leq &c \big(
\| V^m \otimes v^m +   v^{m-1} \otimes V^m + V^{m}\otimes
W+W\otimes V^{m}\|_{L^{p_0} (Q_+)}   \big) \\
\notag \leq &c  \|  V^m\|_{L^{p_0} (Q_+)}  \big( \|  v^m\|_{L^{\infty} (Q_+)} + \|  v^{m-1}\|_{L^{\infty} (Q_+)}  + \|  W\|_{L^{\infty} (Q_+)} \big)\\
\leq &\frac12   \|  V^m\|_{L^{p_0} (Q_+)}.
\end{align}
For $L^{\infty}$ estimate, we compute
\begin{align}\label{0913-2-1}
\notag \|  V^{m+1}\|_{L^{\infty} (Q_+  )}
%\\ \notag
 \leq &c   \big( \| \na V^m\|_{L^q_tL^{r_0} ( Q_+)} \|  v^m\|_{L^{\infty} ( Q_+)} +\| \na V^m\|_{L^q_tL^{r_0} ( Q_+)} \|  W\|_{L^{\infty} ( Q_+)} \big)\\
\leq &\frac12   \|  \na V^m\|_{L^q_tL^{r_0} (Q_+)}.
\end{align}
In the same way as above, we get
\begin{align*}
\| V^m\|_{L^{2q}_tL^{2r_0}_x (Q_+)} & < \frac12\| V^{m-1}\|_{L^{2q}_tL^{2r_0}_x (Q_+)},\\
\| \na V^m\|_{L^{q}_tL^{r_0}_x (Q_+)} & < \frac12\| \na V^{m-1}\|_{L^{q}_tL^{r_0}_x (Q_+)}.
\end{align*}
Therefore, there is $v$ satisfying $v \in L^{2q}_tL^{2r_0}_x (Q_+) \cap L^\infty(Q_+)$, $D_x v \in L^{q}_tL^{r_0}_x (Q_+) \cap L^{p_0}(Q_+)$ and $ D_x^2 v, \, D_t v \in L^{p_0}(Q_+)$ such that
\begin{align*}
&v^m \ri v \quad \mbox{in} \quad L^{2q}_tL^{2r_0}_x (Q_+) \cap L^\infty(Q_+),\\
&D_x v^m \ri D_x v \quad \mbox{in} \quad L^{q}_tL^{r_0}_x (Q_+) \cap L^{p_0}(Q_+),\\
&D_x^2 v^m \ri D_x^2 v \quad \mbox{in} \quad   L^{p_0}(Q_+),\\
&  D_t v^m \ri D_t v  \quad \mbox{in} \quad  L^{p_0}(Q_+).
\end{align*}
Moreover, $v$ solves  \eqref{CCK-Feb7-10-pressure-2} with appropriate pressure $q \in L^{p_0} (Q_+)$.

%$(v^m, \na v^m, \na^2 v^m, D_t v^m)$ converges to $(v, \na v)$ in
%\begin{align*}\label{1130-1}
%\big( L^{p_0}( \R_+ \times (0, 1)) \cap L^{\infty} (\R_+ \times (0, 1))  \big) \times L^{p_0} (\R_+ \times (0, 1)) \times L^{p_0} (\R_+ \times (0, 1)) \times L^{p_0} (\R_+ \times (0, 1))
%\end{align*}

%\[
%v_t-\Delta v+\nabla \Pi=-{\rm div}\,\left(v\otimes v+v\otimes
%W+v\otimes W+W\otimes w\right),
%\]
%\[
%{\rm div} \, v =0,
%\]
%with homogeneous initial and boundary data, i.e. $v(x,0)=0$ and
%$v(x,t)=0$ on $\{x_n=0\}$.

We then set $u:=v+W$ and $p =\pi + q$, which becomes a  weak solution of the
Navier-Stokes equations in $Q_+$, namely
\[
u_t-\Delta u+\nabla p=-{\rm div}\,\left(u\otimes u\right),\qquad
{\rm div} \, u=0 \quad Q^+_1
\]
with boundary data $u(x,t)=g(x,t)$ on $\{x_n=0\}$  such that
\begin{align*}
\|\na^2  u\|_{L^{\frac{3q}2} (Q_1^+)}  + \|D_t  u\|_{L^{q} (Q_1^+)} + \|\na  p\|_{L^{q}. (Q_1^+)}  < \infty.
\end{align*}
Since integrability of spatial variables can be improved, we can have that for given $p>1$
\[
\norm{\na^2 u}_{L^{\frac{3q}2}_t L^{p}_{x}(Q_{1}^+)}+\norm{ D_t u}_{L^{q}_t L^{p}_{x}(Q_{1}^+)}+\norm{\nabla \pi}_{L^{q}_t L^{p}_{x}(Q_{1}^+)}<\infty,
\]
However, it is straightforward via construction that for any  $r_1>q$ and $r_2>\frac{3q}{2}$
%\begin{align*}
%\|D_t  u\|_{L^{r_1} (Q_{\frac{1}{4}}^+)} = \infty, \quad \|\na p\|_{L^{r_1} (Q_{\frac{1}{4}}^+)} = \infty,\quad \|\na^2 u\|_{L^{r_2} (Q_{\frac{1}{4}}^+)} = \infty.
%\end{align*}
\begin{align*}
 \| D_t u \|_{L^{r_1}_t L^{p}_x (Q_{\frac{1}{2}}^+)} =\infty, \qquad \| \nabla \pi \|_{L^{r_1}_t L^{p}_x (Q_{\frac{1}{2}}^+)} =\infty, \qquad \| \nabla^2 u \|_{L^{r_2}_t L^{p}_x (Q_{\frac{1}{2}}^+)} =\infty.
\end{align*}
This completes the proof.
\qed

\appendix
\setcounter{equation}{0}

\section{Alternative proof in Remark \ref{rem-thm11}}
\label{alter-thm11}

%{\color{red}$\bullet$}\footnote{We need to check Appendix A.}

We provide the alternative and simple proof of Theorem \ref{maintheorem-SS}  for the case $p =2$, since it seems informative. Assume that $g_n^{{\mathcal T}} \in L^\infty({\mathbb R}) \setminus \dot H^{\frac14}_{2} ({\mathbb R})$.
Suppose that $w$ is a solution of the Stokes equations
\eqref{Stokes-bvp-200}-\eqref{Stokes-bvp-210} defined by
\eqref{rep-bvp-stokes-w} when $f=0$  and  the boundary data $g$ is given in \eqref{0502-6} with \eqref{boundarydata}. We remind the decompositon of $w = w^{\calL} + w^{\calN}+ w^{\calB,1}+ w^{\calB,2}$, where $ w^{\calL}$, $w^{\calN}$, $w^{\calB,1}$ and $w^{\calB,2}$ are specified in the proof of Proposition \ref{lemma0406}.
We recalled that it was shown in the proof of Proposition \ref{lemma0406}  that
\begin{align*}
\| w^{\calL} \|_{L^\infty(Q_1^+)} + \|  w^{\calN}\|_{L^\infty(Q_1^+)} + \| w^{\calB,2} \|_{L^\infty(Q_1^+)} < \infty,
\end{align*}
\begin{align*}
 \|\na w^{\calL} \|_{L^r(Q_1^+)} + \| \na  w^{\calN}\|_{L^r(Q_1^+)} + \| \na w^{\calB,2} \|_{L^r(Q_1^+)} <\infty \quad 1 < r < \infty.
\end{align*}
The term $w^{\calB,1 }_i$ is represented by
\begin{align*}
   w^{\calB,1 }_i(x,t)
&=   c_n \int_{0}^t D_{x_n} \Ga_1 (x_n, t-s) g^{{\mathcal T}}_n(s)ds \psi_i(x'),
\end{align*}
where $\Ga_1$ is the one dimensional Gaussian kernel and  $\psi_i(x') = \int_{\Rn} \frac{x_i -y_i}{|x' -y'|^n} g_n^{{\mathcal S}}(y') dy'   $ is smooth in $|x' | \leq 1$.
Using the decay of $\Ga_1$ (see \eqref{0730-3-1}), we can obtain
\begin{align}\label{0730-3}
\notag \int_1^\infty  \int_0^\infty       \abs{\int_{0}^t D^2_{x_n} \Ga_1 (x_n, t-s) g^{{\mathcal T}}_n(s)ds}^2           dx_n dt
& \leq  c\| g_n^{\mathcal T}\|_{L^\infty ({\mathbb R})},\\
\int_{1}^\infty  \int_0^\infty      \abs{\int_{0}^t D^2_{x_n} \Ga_1 (x_n, t-s) g^{{\mathcal T}}_n(s)ds}^2          dt     dx_n & \leq c\| g_n^{\mathcal T}\|_{L^\infty ({\mathbb R})}.
\end{align}
Note that   $D_{x_n}^2 \Ga_1(x_n, t) =  x_n^{-1}K_{x_n^2} (t)$, where $K_{x_n^2} (t)= x_n^{-2}K(\frac{t}{x_n^2})$ with
%{\color{red}$\bullet$}\footnote{constant and computation need to be checked again.}
\[
 K(t) = \frac1{\sqrt{4\pi}}\Big(-\frac{4}{t^{\frac32} } + \frac{16}{t^{\frac52}}\Big)
  e^{-\frac{1}{4t}} \chi_{t > 0} = D_t \bke{\frac1{\sqrt{4\pi}} t^{-\frac12}
  e^{-\frac{1}{4t}} \chi_{t > 0}}.
\]
Using  Plancherel Theorem with respect to $t$ and change of variables,  we have
\begin{align}\label{0730-5}
\notag & \int_{-\infty}^{\infty} \int_0^\infty \abs{\int_0^tD^2_{x_n} \Ga_1 (x_n, t-s) g^{{\mathcal T}}_n(s) ds}^2 dx_n dt\\
\notag &  = \int_{-\infty}^{\infty} \int_0^\infty x_n^{-2} |K_{x_n^2} * g_{n}^{\mathcal T}(t)|^2 dx_n dt\\
\notag = &\int_{-\infty}^{\infty}|\hat  g_n^{\mathcal T} (\tau)|^2  \int_0^\infty  x_n^{-2} |\hat K(x_n^2 \tau)|^2  dx_n d \tau\\
\notag  = &\int_{-\infty}^{\infty}|\hat g_n^{\mathcal T}(\tau)|^2|\tau|^{\frac12}   d \tau \int_0^\infty  x_n^{-\frac32} |\hat K(x_n) |^2  dx_n\\
 = &\| g_n^{\mathcal T}\|_{\dot H^{\frac14}_2 ({\mathbb R})} \int_0^\infty  x_n^{-\frac32} |\hat K(x_n) |^2  dx_n.
\end{align}
Since  $\int_{-\infty}^{\infty} K(t) dt =0$, it follows that $|\hat K(x_n) |\leq c |x_n|$ near zero and since $K\in L^1({\mathbb R})$, we have that  $\hat K(x_n)$ is also bounded.
%\footnote{KK: We need $|\hat K(x_n) |<\infty$, don't we?}
 Hence, $\int_0^\infty  x_n^{-\frac32} |\hat K(x_n) |^2  dx_n$  is well-defined.  Thus, from \eqref{0730-5},  we have
\begin{align}\label{0819-5}
\int_{-\infty}^{\infty} \int_0^\infty \abs{\int_0^tD^2_{x_n} \Ga_1 (x_n, t-s) g^{{\mathcal T}}_n(s) ds}^2 dx_n dt
& = c\| g_{n}^{\mathcal T}\|_{\dot H^{\frac14}_2 ({\mathbb R})}.
\end{align}
Since $ \| g_{n}^{\mathcal T}\|_{\dot H^{\frac14}_2 ({\mathbb R})} =\infty$,
it  follows from  \eqref{0730-3}, \eqref{0730-5} and \eqref{0819-5} that
\begin{align*}
\int_{Q^+_1} |D_{x_n} w_1^{\calB, 1}(x,t)|^2 dx dt
\geq c\| g^{\mathcal T}_{n}\|_{\dot H^{\frac14}_2 ({\mathbb R})} \| \psi\|_{L^\infty(B'_1 )} -c\| g_n\|_{L^\infty(\Rn \times (0, \infty))}=\infty.
\end{align*}
Hence, we complete the proof of Remark \ref{rem-thm11}.
\qed

\section{Proof of Remark \ref{rem0711-2}}
\label{appendix0131-0}

The claim in Remark \ref{rem0711-2} is proved in the next lemma.

\begin{lemm}\label{lemma0706-1}
Let
\begin{align*}
\left\{\begin{array}{ll} \vspace{2mm}
 g_n^{{\mathcal T}}(t) &  = \eta(t) (1-t)^\frac12 \ln^{-1} (1-t ) \quad 0 < t < 1,\\
  g_n^{{\mathcal T}}(t)  & =0 \quad  1 \leq t,
 \end{array}
 \right.
 \end{align*}
where $\eta \in  C_c^\infty( \frac34, 2)$ satisfying $\eta \geq 0$ and $\eta = 1$ in $ \frac38 \leq t \leq \frac32$
and $g_n^{\mathcal S} \in C_c^\infty(\Rn) $ satisfy the conditions of \eqref{1222-2}. Then,
\begin{align}\label{0706-10}
\| p \|_{L^2 ( Q_1^+) } + \| \na p\|_{L^2 (Q_1^+)} + \| \nabla w \|_{L^2 (Q_1^+)}  < \infty
\end{align}
but
\begin{align}\label{0706-11}
\| \na w\|_{L^\infty(Q_1^+)} = \infty.
\end{align}

\end{lemm}

\begin{proof}
%{\color{red}$\bullet$}\footnote{Proof is rewritten.}
Note that $ g_n^{{\mathcal T}}$ satisfies the condition in \eqref{1222-2} for $q =2$ without the last one.
From the proof of Theorem \ref{maintheorempressure-SS}, we have
\begin{align*}
\| p \|_{L^2 ( Q_1^+) } + \| \na p\|_{L^2 (Q_1^+)} + \| u \|_{L^2 (Q_1^+)} \leq c \big(\|   g_n^{{\mathcal T}}\|_{L^2(0, 1)} +  \| D_t g_n^{{\mathcal T}}\|_{L^2 (0, 1)} \big) < \infty.
\end{align*}
Hence, \eqref{0706-10} holds.
To prove \eqref{0706-11}, we the decomposition  $w = w^{{\mathcal L}} + w^{{\mathcal B}} + w^{{\mathcal N}}$ defined in Proposition \ref{lemma0406}. In the proof of Proposition \ref{lemma0406},  we have
\begin{align*}
\| \na( w - w^{{\mathcal B},1} )\| _{L^\infty (Q^+_1 )} \leq \| g\|_{L^\infty}.
\end{align*}
Note that  for $\frac34 < t$,
\begin{align*}
 D^2_{x_n}  \Ga_1 * g_n^{\mathcal T} (x_n,t)& = c_n\int_\frac34^t \frac1{(t-s)^\frac12} e^{-\frac{x_n^2}{4(t-s)}} D_s g_n^{{\mathcal T}}(s) ds\\
 & = - c_n\int_\frac34^t \frac1{(t-s)^\frac12} e^{-\frac{x_n^2}{4(t-s)}}\frac{1}{2} (1-s)^{-\frac12} \ln^{-1} (1-s)\bke{1-2\ln^{-1} (1-s)} ds.
\end{align*}
Since $(1-2\ln^{-1} (1-s))\geq  - \ln^{-1} (1-s), \,\, \frac34 < s <1$,    we have
\begin{align*}
| D^2_{x_n}  \Ga_1 * g_n^{\mathcal T} (x_n,1)|& \geq
 c\int_\frac34^1 \frac1{(1-s)^\frac12} e^{-\frac{x_n^2}{4(1-s)}} (1-s )^{-\frac12} |\ln (1-s )|^{-1} ds\\
 & =
 c\int_0^\frac14  s^{-1} e^{-\frac{x_n^2}{4s}}   |\ln s|^{-1} ds\\
 & =
 c\int_{4x_n^2}^\infty s^{-1} e^{-s}   |\ln \frac{x_n^2}{s}|^{-1} ds\\
 & \geq
 c\int_{4x_n^2}^1 s^{-1}   ds\\
 & = -c \ln (4 x_n^2) -c.
\end{align*}
Summing up the above estimates, we have
$\| \nabla w^{{\mathcal B},1} \| _{L^\infty (B_1 \times (0, 1))} =\infty$, depending on the sign of\\
 $\int_{\Rn}    g^{\mathcal S}_n(y')      \frac{x_1 -y_1  }{|x' -y' |^{n-1}}
 dy' ds$, unless it vanishes. It is not difficult to choose $g^{\mathcal S}_n$ such that the integral is not zero.
Thus, we complete the proof of Lemma \ref{lemma0706-1}.
\end{proof}

%\section{proof of Proposition  \ref{theoexternel-boundary-2} }
%\label{appendix0909}

%As defined in \eqref{0909-1}, the solution of \eqref{Stokes-bvp-200}-\eqref{Stokes-bvp-210}  with $g =0$ is represented by
%\begin{align*}
%w(x,t) = \int_0^t \int_{\R_+} K^{\mathcal P}(x,y, t-s) {\mathbb P} f(y,s) dyds.
%\end{align*}
%From \eqref{0801-1} and Holder inequality, for $p > \frac{n+2}2$, we have
%\begin{align*}
%|w(x,t)| &  \leq c\int_0^t (t -s)^{-\frac{n}{2p}} \|  f(s)\|_{L^p (\R_+)} ds\\
%&  \leq c \big( \int_0^t (t -s)^{-\frac{n}{2p}\frac{p}{p-1}}  ds \big)^{\frac{p-1}p} \|  f\|_{L^p (\R_+ \times (0, t))}\\
%&  \leq c t^{1-\frac{n+2}{2p}} \|  f\|_{L^p (\R_+ \times (0, t))}.
%\end{align*}
%Hence, we complete the proof of Proposition  \ref{theoexternel-boundary-2}.
%{\color{red}$\bullet$} \footnote{It seems mismatch for the exponent.}

\section{Examples in  Remark \ref{rem0314} and Remark \ref{example-0924}}
\label{appendix0131}

$\bullet$ ($ g^{\mathcal T}_n \in L^\infty({\mathbb R}) \setminus \dot B^{\frac12 -\frac1{2p}}_{pp} ({\mathbb R})$, $1<p<\infty$ in Remark \ref{rem0314})

Let $ 0< a < 1 $ be a number satisfying $3 -\frac2{1 +a} < p$. Define $g_n^{\mathcal T}(t)$  by
\begin{align*}
g_n^{\mathcal T}(t) = \sum_{k =1}^\infty \chi_{( (2k +1)^{-a}, (2k)^{-a})  } (t),
\end{align*}
where $\chi$ is a characteristic function. It is direct that $g_n^{\mathcal T} \in L^\infty (0,1)$.
We claim that $ g_n^{\mathcal T} \notin \dot B^{\frac12 -\frac1{2p}}_{pp} (0, 1)$. Indeed,
\begin{align}\label{0928-1}
\notag \int_0^1 \int_0^1 \frac{|g_n^{\mathcal T}(t) - g_n^{\mathcal T}(s)|^p}{|t -s|^{1 + p (\frac12 -\frac1{2p})}} dsdt &\geq \sum_{ k \geq 1}  \int_{ (2k +1)^{-a} }^{(2k )^{-a}  } \int_{(2k )^{-a}  }^1 \frac{|1 - g_n^{\mathcal T}(s)|^p}{|t -s|^{1 + p (\frac12 -\frac1{2p})}} dsdt\\
& \geq  \sum_{ k \geq 1} \sum_{1 \leq l \leq k } \int_{ (2k +1)^{-a} }^{(2k )^{-a}  } \int_{(2l)^{-a}}^{(2l -1)^{-a}} \frac{1}{|t -s|^{\frac12 +\frac{p}{2}}} dsdt.
\end{align}
For $ t \in  ( (2k +1)^{-a},(2k )^{-a})$ and $l < k$, there is $c_1 > 0$  independent of $k, \, l$ such that
%{\color{red}$\bullet$}\footnote{$c_a=1$?}
\begin{align*}
\int_{(2l)^{-a}}^{(2l -1)^{-a}} \frac{1}{|t -s|^{\frac12 +\frac{p}{2}}} ds
\geq  \int_{(2l-1)^{-a}}^{(2l -2)^{-a}} \frac{1}{|t -s|^{\frac12 +\frac{p}{2}}} ds.
\end{align*}
This implies that
%{\color{red}$\bullet$}\footnote{$c_a=1/2$?}
\begin{align*}
\sum_{1 \leq l \leq k } \int_{(2l)^{-a}}^{(2l -1)^{-a}} \frac{1}{|t -s|^{\frac12 +\frac{p}{2}}} ds
& \geq \frac{1}{2}  \int_{(2k )^{-a}  }^1 \frac{1}{|t -s|^{\frac12 +\frac{p}{2}}} dsdt\\
&  =c  \bke{ \frac{1}{( (2k)^{-a} -t)^{-\frac12 +\frac{p}{2}}} - \frac{1}{(1-t)^{-\frac12 +\frac{p}{2}}} }\\
&  \geq c   \frac{1}{( (2k)^{-a} -t)^{-\frac12 +\frac{p}{2}}}.
\end{align*}

\begin{align}\label{0928-2}
\notag \sum_{1 \leq l \leq k } \int_{ (2k +1)^{-a} }^{(2k )^{-a}  } \int_{(2l)^{-a}}^{(2l -1)^{-a}} \frac{1}{|t -s|^{\frac12 +\frac{p}{2}}} dsdt
&  \geq c\int_{ (2k +1)^{-a} }^{(2k )^{-a}  }   \frac{1}{( (2k)^{-a} -t)^{-\frac12 +\frac{p}{2}}} dt\\
&   = \left\{ \begin{array}{l}
\infty \quad \mbox{if } \quad p \geq 3,\\
c_4  ((2k)^{-a} - (2k +1)^{-a} )^{\frac32 -\frac{p}{2}} \quad \mbox{if } \quad 1 < p < 3.
\end{array}
\right.
\end{align}
%{\color{red}$\bullet$}\footnote{second inequality is opposite, isn't it?}
Hence, $g_n^2 \notin \dot B^{\frac12 -\frac1{2p}}_{pp} (0, 1)$ for $ p \geq 3$ and $a > 0$.
Using mean-value Theorem, there is $\xi \in (2k, 2k +1)$,  we have
\begin{align}\label{0928-3}
(2k)^{-a} - (2k +1)^{-a} = \frac{(2k +1)^a - (2k)^a}{(2k)^{a} (2k +1)^{a}  }
= \frac{a \xi^{a-1} }{(2k)^{a} (2k +1)^{a}  } \geq c k^{-1 -a}.
\end{align}
By \eqref{0928-1}, \eqref{0928-2} and \eqref{0928-3}, for $ 1 < p < 3$, we have
\begin{align}\label{0314-1}
\notag \int_0^1 \int_0^1 \frac{|g_n^{\mathcal T}(t) - g_n^{\mathcal T}(s)|^p}{|t -s|^{1 + p (\frac12 -\frac1{2p})}} dsdt
&\geq    c_4 \sum_{ k \geq 1}    ((2k)^{-a} - (2k +1)^{-a} )^{\frac32 -\frac{p}{2}}\\
& \geq    c_5 \sum_{ k \geq 1}   k^{-(1 +a)(\frac32 -\frac{p}2) }.
\end{align}
Since $-(1 +a)(\frac32 -\frac{p}2)  > -1$, the right-hand side of \eqref{0314-1} is infinite. Therefore, $g_n^{\mathcal T} \notin \dot B^{\frac12 -\frac1{2p}}_{pp} (0, 1)$.
\qed
\\
\\
%\section{Example of Remark \ref{example-0924}}
%\label{appendix0917}
$\bullet$ ($  D_t g^{\mathcal T}_n \in L^q( 0,1) \setminus  L^r(0,1)$ for $q < r$ and $  g^{\mathcal T}_n \notin B^{1 -\frac1{2r}}_{rr} (0,1) $ for  $r>\frac{3q}2$ in Remark \ref{example-0924})

Direct calculations show %that $ D_t g_n^{\mathcal T} (t) = 0$ for $0 < t \leq  \frac34 $ and $D_t g_n^{\mathcal T} (t) =      (1 -\frac1q ) (t -\frac3{4})^{ -\frac1q} -      (t -\frac3{4})^{ -\frac1q} \ln (t-\frac3{4})^{-2}$ for $ \frac3{4} < t < 1$.
%Hence, we have
\begin{align*}
\| D_t g_n^{\mathcal T}\|^q_{L^q (0, 1)} & \leq c \int_\frac34^1 (t-\frac34)^{-1}|\ln (t-\frac34)|^{-q}dt < \infty,\\
\| D_t g_n^{\mathcal T}\|^r_{L^q (0, 1)} & \geq c \int_\frac34^1 (t-\frac34)^{-\frac{r}q}|\ln (t-\frac34)|^{-r}dt  = \infty \quad r > q.
\end{align*}
Reminding the definition of Beosv space, we have
\begin{align*}
\| g_n^{\mathcal T}\|^r_{\dot B^{1 -\frac1{2r}}_{rr} (0,1)} & = \int_0^1 \int_0^1 \frac{| g_n^{\mathcal T} (t) -  g_n^{\mathcal T} (s)|^r }{ |t -s|^{1 + r ( 1 -\frac1{2r})}} dsdt \\
& \geq \int_{\frac34}^1  \int_{\frac34}^t  \frac{| g_n^{\mathcal T} (t) -  g_n^{\mathcal T} (s)|^r}{|t -s|^{1 + r ( 1 -\frac1{2r})}  }   dsdt.
\end{align*}
Using mean-value Theorem, for $\frac34 < t <1 $ and $\frac34 < s < t$,  there is $\xi \in (s, t)$ such that $|g_n^{\mathcal T} (t) -  g_n^{\mathcal T} (s)| = |D_t g_n^{\mathcal T} (\xi)|  |t-s| \geq c (t-\frac34)^{-\frac1q} |\ln (t -\frac34)|^{-1} |t-s|$. Hence, we have
\begin{align*}
\| g_n^{\mathcal T}\|^r_{\dot B^{1 -\frac1{2r}}_{rr} (0,1)}  & \geq c\int_{\frac34}^1  (t -\frac34)^{ -\frac{r}q } |\ln (t-\frac34)|^{-r} \int_{\frac34}^t  |t-s|^{-\frac12}   dsdt\\
& \geq c\int_{\frac34}^1  (t-\frac34)^{ -\frac{r}q  +\frac12} |\ln (t-\frac34)|^{-r} dt\\
& = \infty \quad \mbox{for} \quad  r>\frac{3q}2.
\end{align*}
Therefore, this completes to provides an example mentioned Remark 6.1.
\qed

\section*{Acknowledgements}
T. Chang is partially supported by NRF-2020R1A2C1A01102531.  K. Kang is partially
supported by  NRF-2019R1A2C1084685. We are grateful to Chan Hong Min for his careful reading and useful comments.

%\begin{equation*}
%\left.
%\begin{array}{l}
%{ \mbox{Tongkeun Chang}}: \,{\mbox{chang7357@yonsei.ac.kr}}\\
%{\mbox{Kyungkeun Kang}}: \,{\mbox{kkang@yonsei.ac.kr }}
%\end{array}
%\right.
%\end{equation*}

\end{document}